\documentclass[10pt]{amsart}

\usepackage[usenames,dvipsnames]{xcolor}
\usepackage[bookmarks,colorlinks=true,citecolor=OliveGreen,linkcolor=RoyalBlue]{hyperref}
\usepackage{amssymb,amsthm}
\usepackage{mathtools}
\usepackage{mathabx}
\usepackage{xparse}
\usepackage{thmtools}
\usepackage[capitalise]{cleveref}		
\usepackage{bm}
\usepackage{comment}

\usepackage{color,soul}
\definecolor{ItalianApricot}{rgb}{1,0.7,0.5}
\sethlcolor{ItalianApricot}

\makeatletter
\def\thmt@refnamewithcomma #1#2#3,#4,#5\@nil{%
  \@xa\def\csname\thmt@envname #1utorefname\endcsname{#3}%
  \ifcsname #2refname\endcsname
    \csname #2refname\expandafter\endcsname\expandafter{\thmt@envname}{#3}{#4}%
  \fi
}
\makeatother

\declaretheorem[numberwithin=section]{theorem}
\declaretheorem[sibling=theorem]{proposition}
\declaretheorem[sibling=theorem]{corollary}

\declaretheorem[sibling=theorem]{lemma}

\declaretheorem[sibling=theorem,style=definition]{definition}

\declaretheorem[sibling=theorem,style=remark]{remark}
\declaretheorem[sibling=theorem,style=remark]{example}

\declaretheorem[sibling=theorem,style=remark]{fact}

\newcommand{\Z}{\mathbb{Z}}
\newcommand{\twist}{\operatorname{twist}}
\newcommand{\Span}{\operatorname{Span}}

\newcommand{\seq}[1]{{\left\langle{#1}\right\rangle}}
 
\newcommand{\rest}[1]{\! \upharpoonright_{#1}} 
\newcommand{\tth}{{}^{\textup{th}}}

\DeclareMathOperator{\range}{range}

\DeclareMathOperator{\cf}{cf}

\newcommand{\cero}{\mathbf{0}}

\newcommand{\diverge}{\!\!\uparrow}

\newcommand{\w}{\omega}

\newcommand{\vphi}{\varphi}

\renewcommand{\le}{\leqslant}
\renewcommand{\ge}{\geqslant}

\newcommand{\nle}{\nleqslant}

\newcommand{\Int}{\mathbb Z}
\newcommand{\Rat}{\mathbb Q}
\newcommand{\Tur}{\textup{\scriptsize T}}
\newcommand{\Nat}{\mathbb N}

\renewcommand {\div}{\!\divides\!}
\newcommand{\ndiv}{\notdivides}

\newcommand{\degd}{\mathbf{d}}

\newcommand{\sing}[1]{\textup{s}(#1)}
\newcommand{\Detach}[1]{\textup{Div}(#1)}

\DeclareMathOperator{\ZF}{ZF}

\newcommand{\PP}{\mathbb P}
\newcommand{\HH}{\mathcal H}

\newcommand \BlackBox[2]{\textup{BlackBox}(#1,#2)}

\newcommand \bdd{\textup{bdd}}
\newcommand \FDet{\textup{Div}_\bdd}

\title[Finding bases of free abelian groups is usually difficult]{Finding bases of uncountable free \\ abelian groups is usually difficult}

\author[N.~Greenberg]{Noam Greenberg} 
\address{School of Mathematics and Statistics, Victoria University of Wellington, Wellington, New Zealand}
\email{greenberg@msor.vuw.ac.nz}
\urladdr{\url{http://homepages.mcs.vuw.ac.nz/~greenberg/}}

\author[D.~Turetsky]{Dan Turetsky} 
\address{School of Mathematics and Statistics, Victoria University of Wellington, Wellington, New Zealand}

\author[L.B.~Westrick]{Linda Brown Westrick} 
\address{School of Mathematics and Statistics, Victoria University of Wellington, Wellington, New Zealand}

\thanks{Greenberg was supported by the Marsden Fund, a Rutherford Discovery Fellowship from the Royal Society of New Zealand and by the Templeton Foundation via the Turing centenary project ``Mind, Mechanism and Mathematics''. Westrick was supported by the Rutherford Discovery Fellowship as a postdoctoral fellow.}

\begin{document}

\begin{abstract}
	We investigate effective properties of uncountable free abelian groups. We show that identifying free abelian groups and constructing bases for such groups is often computationally hard, depending on the cardinality. For example, we show, under the assumption $V=L$, that there is a first-order definable free abelian group with no first-order definable basis. 
\end{abstract}

\maketitle

\section{Introduction}

How complicated is it to find a basis of a free abelian group? Can it be done recursively, as we do when building bases for vector spaces? Here by a basis we mean a subset which is both linearly independent and spans the whole group (with integer, rather than rational coefficients). The difficulty is that unlike vector spaces, free abelian groups can contain maximal linearly independent subsets which are not bases. For countable groups, there is a 
strengthening of linear independence, originally used by Pontryagin~\cite{Pontryagin}, which allows us to recover a recursive construction. This notion generalises $p$-independence, which is widely used in the study of torsion-free abelian groups. Recall that a subgroup~$H$ of a torsion-free abelian group~$G$ is \emph{pure} if $G\cap \Rat H = H$; that is, if for all~$n\in \Int$ and all~$h\in H$, if $n$ divides~$h$ in~$G$ then it also divides it in~$H$. 

\begin{definition}
\label{def:P-independent}
	Let~$G$ be a torsion-free abelian group. A subset $A\subseteq G$ is \emph{$P$-independent} if it is linearly independent and its span is a pure subgroup of~$G$. 
 \end{definition}

Note that any subset of a $P$-independent set is also $P$-independent. Let $\Z^\w = \bigoplus_{k\in \Nat} \Z$ denote the countably generated free abelian group. The following is implicit in Pontryagin's work, and is stated in the following way, for example, in Downey and Melnikov's \cite{DowneyMelnikov:EffectivelyCategoricalGroups} (who generalised it to completely decomposable groups).

\begin{proposition}
\label{prop:extension_of_finite_P-independece}
	Suppose that $B\subset \Z^\w$ is a finite $P$-independent subset; let $g\in \Z^\w$. Then there is a finite $P$-independent $B'\supseteq B$ such that $g\in \Span(B)$. 
\end{proposition}

Again, to be specific, $\Span(B)$ is the set of elements of~$G$ of the form $\sum m_i a_i$ where $a_i\in B$ and $m_i\in \Int$; $B$ is a basis of~$\Int^\w$ if it is linearly independent and spans~$\Int^\w$, if and only if $\Int^\w = \bigoplus_{b\in B} \Int b$. Of course every basis of $\Int^\w$ must be $P$-independent.  \Cref{prop:extension_of_finite_P-independece} tells us that a basis for~$\Int^\w$ can be built recursively, repeatedly extending finite $P$-independent subsets while ensuring that the next element of the group (in some arbitrary $\w$-enumeration of the elements of the group) belongs to the span of the basis that we are building. 

\smallskip

Can this process be mimicked when we are given an uncountable free abelian group? We know that there is no important difference between countable and uncountable vector spaces. A basis for a vector space can be built by transfinite recursion, extending as usual at successor steps and taking unions at limit stages. Searching the literature, we found no such construction for uncountable free abelian groups. The purpose of this paper is to show that in most cases, such a construction cannot be performed. One key point is that \cref{prop:extension_of_finite_P-independece} heavily relies on the fact that~$B$ is finite.\footnote{This is a common theme in the investigation of effective properties of uncountable objects: there is a significant difference between finiteness and boundedness. See for example \cite{Memoirs,LO1}.} A recursive construction can get stuck at a limit stage: we can find elements $a_1,a_2,\dots, $ of a free abelian group~$G$ such that each finite initial segment $\{a_1,a_2,\dots, a_n\}$ can be extended to a basis of~$G$, but such that the countable set $\{a_1,a_2,\dots\}$ cannot be extended to a basis of~$G$. We thank Alexander Melnikov for pointing out to us the following construction, which was known to Fuchs. 

\begin{example}
\label{exmp:Melnikov}
	Let~$G = \Z^{\w+1}$ be a copy of the countably generated free abelian group, with a basis reordered in order-type~$\w+1$: let $\{e_0,e_1,e_2,\dots, e_\w\}$ be a basis of~$G$. For $i<\w$ let $u_i = p_i e_i + e_\w$, where $p_0,p_1,\dots$ is an enumeration of the prime numbers. 

	The set $U = \{u_0,u_1,\dots\}$ is $P$-independent: if $p_j\div \sum a_ip_i e_i + \sum a_i e_\w$ then $p_j \div a_i p_i$ and $p_j \div \sum a_i$; it follows that $p_j\div a_i$ for all $i\ne j$, and so must also divide~$a_j$. \Cref{prop:extension_of_finite_P-independece} implies that any finite subset of~$U$ can be extended to a basis of~$G$. However,~$U$ cannot be extended to a basis of~$G$. Suppose otherwise. Extending to a basis and taking a finite subset, we can find a finite set~$V\subset G$ such that $e_\w\in \Span(V)$ and $V\cup U$ is $P$-independent. There is some~$n$ such that $V$ is spanned by $ \{ e_0, e_1,\dots,e_\w\}\setminus \{e_n\}$. Let $H = \Span(V\cup \{u_n\})$. We show that~$H$ is not pure, contradicting the $P$-independence of $V\cup U$. This is because $p_ne_n  = u_n-e_\w$ is an element of~$H$; but $e_n\notin H$. In fact for any $h\in H$, if $h = \sum_{i\le \w} \alpha_i e_i$ then $p_n \div \alpha_n$; $u_n$ is the only generator that can contribute anything in the $n\tth$ standard coordinate.\footnote{More formally: $h = b u_n + \sum b_i v_i$ for $v_i\in V$, and for each~$i$, $v_i = \sum_{j\le \w} c_{i,j} e_j$, with $c_{i,n} = 0$ for all~$i$.}
\end{example}

Of course, one could imagine that there is another property, even stricter than $P$-independence, adherence to which will allow us to pass limit stages without breaking down. We show that there cannot be any such property. 

\smallskip

What do we actually mean by that statement? If~$G$ is a free abelian group then there \emph{is} a property~$Q$ of subsets of~$G$ (say of smaller cardinality than~$G$) such that: 
\begin{itemize}
	\item Every subset satisfying~$Q$ is linearly independent (or even $P$-independent);
	\item The analogue of \cref{prop:extension_of_finite_P-independece} holds: for every subset~$A$ satisfying~$Q$ and every $g\in G$ there is some $A'\supseteq A$ satisfying~$Q$ such that $g\in \Span(A)$; and
	\item If $A = \bigcup_{\alpha<\lambda} A_\alpha$ is a union of subsets satisfying~$Q$, then $A$ satisfies~$Q$ as well. 
\end{itemize}
Namely, we can let $Q$ hold of the subsets of a fixed basis~$B$ of~$G$. And in turn, we can use~$Q$ to ``recursively'' build~$B$. What we mean by the statement above is that there is no way to obtain such a property~$Q$ if we are just handed the group table for~$G$ and don't have a basis to begin with. Informally, we want to show that it is impossible to only use the group operation of a free abelian group to build a basis. 

To make this statement formal we use the tools of mathematical logic, in particular definability and computability. We fix an uncountable cardinal~$\kappa$ and show that 
\begin{itemize}
	\item For any $\Delta^1_1$ $\kappa$-Turing degree~$\degd$, there are $\kappa$-computable free abelian groups~$G$ with no $\degd$-computable basis. In particular, there is a first-order definable free abelian group of size~$\kappa$ with no first-order definable basis. 
\end{itemize}
Before we explain further, we state two important caveats. The first is that this non-definability result holds for most cardinals $\kappa$ but not for all of them. It is known to fail at some singular cardinals, such as $\aleph_\w$. Even among regular cardinals, we do not know how to show this for weakly compact ones. The second caveat is that throughout we make a non-trivial set-theoretic assumption: that all sets are constructible. While this is often harmless when uncountable computability is concerned, it does leave open the possibility that the picture is different under other, possibly strong, set-theoretic assumptions.

\subsection*{Groups in computable algebra and set theory} 
\label{sub:computable_model_theory}

The study of effective procedures in group theory goes back to work of Max Dehn \cite{Dehn} on finitely presented groups, and in fields, rings and vector spaces to work of Hermann \cite{Hermann}, van der Waerden \cite{VanDerWaerden}, and explicitly using computability to Rabin \cite{Rabin}, Maltsev \cite{Malcev}, Fr\"{o}hlich and Shepherdson \cite{FS56}, and Metakides and Nerode (for example \cite{MetakideRecSpaces}). The basic idea is to study how effective algebraic objects and processes are. For example, famously, Novikov and Boone (see for example \cite{Novikov,Boone}) showed that the word problem in groups may fail to be solved effectively; the same holds for conjugacy and isomorphism questions. Similarly, Higman's embedding theorem \cite{Higman} characterises embeddability into finitely presented groups using an effective criterion. 

The key notion is that of a \emph{computable group}: this is one whose collection of elements is a computable set (say of natural numbers), and the group operation can be performed effectively (computably). Key questions are: (a) which groups have computable copies? and (b) how similar or different are various computable copies of the same group? One possible answer for the second question is encapsulated in the notion of \emph{computable categoricity}, meaning that all computable copies are isomorphic via computable isomorphisms; informally, this means that all computable copies have the same computable properties. For example, finitely generated free abelian groups are computably categorical, since a bijection between two finite bases effectively lifts to an isomorphism of the groups. 

Very few groups are computably categorical, and so it makes sense to consider weakenings of this notion by allowing the help of the jump operator. For example, we say that a group is \emph{$\Delta^0_2$-categorical} if any two computable copies are isomorphic via a $\Delta^0_2$ ($\cero'$-computable) isomorphism. For a free abelian group, the complexity of isomorphisms with a standard computable copy (with a computable basis) is the same as the complexity of bases. In~\cite{DowneyMelnikov:EffectivelyCategoricalGroups}, Downey and Melnikov use \cref{prop:extension_of_finite_P-independece} to show that the countably generated free abelian group is $\Delta^0_2$-categorical, equivalently, that every computable copy of the countably generated free abelian group has a~$\Delta^0_2$ basis.\footnote{This is sharp: in this paper we show that there is a computable copy of $\Z^\w$, every basis of which computes $\emptyset'$.} 

\smallskip

Uncountable free groups were studied by algebraists and set theorists. Best known is Shelah's work on the Whitehead problem \cite{Shelah:Whitehead}. Two main questions which were addressed are: (1) for which cardinals~$\lambda$ are there $\lambda$-free groups which are not $\lambda^+$-free? (2) Is it possible to axiomatise the class of free abelian groups in infinitary logic? The latter question is related to results below on the complexity of the collection of free abelian groups of a fixed cardinality. Some techniques used for the investigation of these questions are related to ones we use below. These investigations though were not concerned with questions of effectiveness. See for example \cite{Hill:NewCriteriaI,Eklof:FreenessAtRegulars,Shelah:SingularCompactness} and the book \cite{EklofMekler}.


\subsection*{Uncountable computable algebra} 
\label{sub:uncountable_computable_model_theory}

The tools of traditional computability are restricted to investigating countable groups, since the basic objects that can be manipulated by computers are hereditarily finite. To be able to make sense of the questions above for uncountable groups we use an extension of computability to uncountable domains. Several approaches have been suggested (see \cite{EMUbook}). In this paper we use \emph{admissible computability}, as described in \cite{GreenbergKnight}, to investigate uncountable computable model theory. This approach was successfully used in \cite{LO1,LO2} to investigate uncountable linear orderings. An abstract investigation of computable categoricity in this setting is given in \cite{Julia:Arithmetic,UnfinishedCC}. 

There are several ways to describe admissible computability. K\"{o}pke~\cite{Koepke:OrdinalComputability} used Turing machines with transfinite tapes. The original way, and quickest, is to use definability. Let~$\kappa$ be a cardinal. The universe for $\kappa$-computability is $H(\kappa)$, the collection of all sets whose transitive closure is of size smaller than~$\kappa$. A set is defined to be \emph{$\kappa$-c.e.}\ if it is $\Sigma^0_1$-definable over $H(\kappa)$ (with parameters). A set is $\kappa$-computable if it is both $\kappa$-c.e.\ and co-$\kappa$-c.e.; a function is partial $\kappa$-computable if its graph is $\kappa$-c.e. The main assumption which makes computability work is that there is a $\kappa$-computable isomorphism between~$\kappa$ and the universe~$H(\kappa)$. Most commonly this is achieved by assuming that every set in~$H(\kappa)$ is constructible, in which case $H(\kappa)= L_\kappa$. Note that this holds for $\kappa=\w$, and that $\w$-computability is the familiar notion of Turing computability. The main tool in $\kappa$-computability is defining computable functions recursively. Formally, if $I\colon H(\kappa)\to H(\kappa)$ is $\kappa$-computable then there is a unique function $f\colon \kappa\to H(\kappa)$ such that for all $\alpha<\kappa$, $f(\alpha) = I(f\rest \alpha)$; this function is $\kappa$-computable. The main point is that even when $\kappa$ is singular, $f\rest \alpha \,\in H(\kappa)$; we say that~$H(\kappa)$ is \emph{admissible}. 

For more details on $\kappa$-computability see \cite{SacksBook,GreenbergKnight}. As we mentioned above, throughout this paper we assume that $V=L$.


\subsection*{Identifying free groups} 
\label{sub:identifying_free_groups}

When investigating the complexity of free abelian groups we come across a closely related question: how complicated is it to tell if a given (torsion-free and abelian) group is free? Informally, the idea is that if there were some effective or definable way to take the group operation of a free group and produce a basis, we could start with any group, attempt to build a basis according to this procedure, and see if we succeed or fail. For example, for countable groups this approach, using the Downey-Melnikov procedure described above, gives an upper bound for the complexity of the collection of free groups (it is $\Pi^0_3$, in fact complete at that level). 
On the other hand it stands to reason that a procedure that tells whether a given group is free can be used to get a proof of this fact, namely a basis. Thus the complexity of the two problems, of identifying free groups, and of building bases, is often related. We shall see though that in some cases this intuition does not seem to hold.

There is a natural upper bound to the complexity of the collection of free groups, namely $\Sigma^1_1$ --- the defining formula is ``the group has a basis''. A proof that this collection is $\Sigma^1_1$-complete would show that there is no simpler way of identifying free groups; a proof that this collection is much simpler (say first-order definable) would show that there is some kind of effective or definable procedure to find out whether a group is free, without having to divine a basis out of nowhere. Our first theorem settles the complexity of the collection of free abelian groups for regular uncountable cardinals. 

\begin{theorem}[$V=L$] \label{thm:identifying_free_groups}
	Let~$\kappa$ be a regular uncountable cardinal. 
	\begin{enumerate}
		\item If $\kappa$ is not weakly compact then the collection of free abelian groups is $\mathbf{\Sigma}^1_1$-complete. If further~$\kappa$ is a successor cardinal, or the least inaccessible cardinal, then this collection is $\Sigma^1_1$-complete. 
		\item If~$\kappa$ is weakly compact then the collection of free abelian groups is ${\mathbf{\Pi}}^0_1$-complete (indeed it is $\Pi^0_1(\emptyset')$-complete) in the set of groups. The index-set\footnote{Using a $\kappa$-computable listing $\seq{W_\alpha}_{\alpha<\kappa}$ of all $\kappa$-c.e.\ sets (obtained from a universal $\Sigma^0_1(L_\kappa)$ predicate), the notions of an index for a $\kappa$-computable object and of an index-set are identical to the familiar one from Turing computability. 
} of the computable free abelian groups is $\Pi^0_2$-complete. 
	\end{enumerate}
\end{theorem}

We should be more formal about what we actually mean. Just as for $\kappa=\w$, if $\kappa$ is regular then we can discuss the complexity of subsets of $2^\kappa$ (or $\kappa^\kappa$) using definability. The subset of $2^\kappa$ defined by a formula~$\vphi$ (in the language of set theory) is the collection of $A\subseteq \kappa$ such that the structure $(H(\kappa);\in, A)$ satisfies~$\vphi$. We allow quantification over subsets of~$\kappa$; for example, a $\Sigma^1_1$ formula $\exists X\,\vphi$ holds of~$A$ if there is some $B\subseteq \kappa$ such that $(H(\kappa);\in,A,B)\models \vphi$. In all of these formulas we allow parameters from the structure $H(\kappa)$. We also use the usual conventions for boldface classes, to denote that we allow a fixed subset predicate. 

Lightface statements of completeness are effective. We use a rich topology for~$2^\kappa$: basic open sets are specified by specifying fewer than~$\kappa$ bits. A partial continuous function from~$2^\kappa$ to itself is defined by a \emph{functional}, a set~$\Phi$ of pairs $(p,q)$ where $p,q\in 2^{<\kappa}$, and satisfying the requirement that if $(p,q), (p', q') \in \Phi$ and $p$ and $p'$ are comparable, then $q$ and $q'$ are also comparable; the defined function maps $A\in 2^\kappa$ to~$\Phi(A)$ defined by $q\prec \Phi(A)$ if and only if there is some $p\prec A$ such that $(p,q)\in \Phi$. If~$\Phi$ itself is $\kappa$-c.e.\ then the induced function is called partial $\kappa$-computable. We remark that just as
in the case $\kappa=\w$, relative $\kappa$-computability can be defined using these maps; we say that $X\in 2^\kappa$ is $\kappa$-computable from~$Y\in 2^\kappa$ if $f(Y)=X$ for some partial $\kappa$-computable~$f$.


When we say that the collection of free abelian groups is $\Sigma^1_1$ 
complete, we mean that for any $\Sigma^1_1$ set~$R\subseteq 2^\kappa$ there is a $\kappa$-computable function~$f\colon 2^\kappa\to 2^\kappa$ such that for all $A\in 2^\kappa$, $A\in R$ if and only if $f(A)$ is (the graph of the group operation of) a free abelian group. This also gives an index-set result: it shows that the collection of indices of partial $\kappa$-computable functions 
$f:\kappa\rightarrow\kappa$ which are total and compute (the graph of the group operation of) a free abelian group is complete among all~$\Sigma^1_1$ subsets of~$\kappa$.

By boldface completeness we mean to allow an oracle. That is, $\mathbf{\Sigma}^1_1$-completeness stated above says that there is some~$A\subseteq \kappa$ such that the collection of free abelian groups is $\Sigma^1_1(A)$, and for any $\Sigma^1_1(A)$ set~$R$ there is an $A$-computable function~$f$ which reduces~$R$ to the collection of free abelian groups.  

We also remark that the first part of \cref{thm:identifying_free_groups} can be relativised to any oracle. Namely, if the collection of free abelian groups of size~$\kappa$ is $\Sigma^1_1(A)$-complete for some $A\in 2^\kappa$, then it is also $\Sigma^1_1(B)$-complete for all~$B\in 2^\kappa$ which $\kappa$-compute~$A$.

We remark though that when we later discuss singular cardinals we cannot relativise to any oracle, as for many oracles~$A$ the structure $(H(\kappa);\in,A)$ will not be admissible.


\subsection*{The complexity of bases} 
\label{sub:the_complexity_of_bases}

\Cref{thm:identifying_free_groups} gives us information about the complexity of bases of free groups. The fact that there is a complete $\Sigma^1_1$ subset of~$\kappa$ implies:

\begin{corollary}[$V=L$] \label{cor:no_upper_bound_for_bases}
	Let~$\kappa$ be a successor cardinal. For any $\Delta^1_1$ set~$X\in 2^\kappa$ there is a $\kappa$-computable free abelian group which has no $X$-computable basis. 
\end{corollary}

We note that the class $\Delta^1_1(L_\kappa)$ is huge. It properly contains all $\kappa$-hyperarithmetic sets (under any reasonable definition of that concept), all sets in the least admissible set beyond $L_\kappa$ (or the least model of $\textrm{ZF}^-$), and more. 

\smallskip

One could hope for more. Can we not only avoid lower cones but code complicated information into all bases of a group? We will show that this is not the case; bases can be built by forcing and so can avoid computing even simple sets. 

\begin{theorem}[$V=L$] \label{thm:coding_into_bases}
	Let~$\kappa$ be a regular uncountable cardinal.
	\begin{itemize}
		\item If $\kappa$ is a successor of a regular uncountable cardinal which is not weakly compact, 
		let $D = \emptyset''$ (the complete $\Sigma^0_2(L_\kappa)$ set). 
		\item Otherwise 
		let $D = \emptyset'$ (the complete $\Sigma^0_1(L_\kappa)$ set). 
	\end{itemize}
	Then:
	\begin{enumerate}
		\item There is a $\kappa$-computable free abelian group, all of whose bases $\kappa$-compute~$D$. 
		\item If $X \nle_{\kappa} D$ then every $\kappa$-computable free abelian group has a basis which does not $\kappa$-compute~$X$. 
	\end{enumerate}
\end{theorem}

In fact coding in~$\emptyset'$ is not complicated; we will show that for any cardinal~$\kappa$ there is a $\kappa$-computable free abelian group, all of whose bases compute~$\emptyset'$. The proof covers $\kappa=\w$ and singular cardinals as well. 


\subsection*{Singular cardinals} 
\label{sub:singular_cardinals}

Singular cardinals pose many difficulties. If~$\kappa$ is singular then for many sets~$A$, $(L_\kappa,A)$ is not admissible, and computability itself behaves in strange ways. For example, the $\aleph_{\w_1}$-degrees above $\emptyset'$ are well-ordered \cite{Friedman:AlephOmega1}. However~$L_\kappa$ itself is admissible and in some cases we can say something about $\kappa$-computable groups. For example, we can code $\emptyset'$ into bases of a group. In the case of cofinality~$\w$, the complexity introduced by closed and unbounded sets disappears, and we can say more. 

\begin{theorem}[$V=L$] \label{thm:singular_case} \
	\begin{enumerate}
		\item Every $\aleph_\w$-computable group has a $\emptyset'$-computable basis. 
		\item The index set of the $\aleph_\w$-computable free abelian groups is $\Pi^0_2$-complete. 
	\end{enumerate}
\end{theorem}

A more general theorem holds for all cardinals of cofinality~$\w$. 

\subsection*{Questions} 
\label{sub:questions}

We are left with several questions.
\begin{enumerate}
	\item Can \cref{cor:no_upper_bound_for_bases} be strengthened? For example, is there a $\kappa$-computable group with no $\Delta^1_1(L_\kappa)$ basis? We remark that for regular uncountable cardinals there is no ``overspill'' phenomenon. 
	\item There are cases which were not covered. For example, we don't know 
	if \cref{cor:no_upper_bound_for_bases} holds for weakly compact cardinals. 
	\item What happens if $V\ne L$? Recall that for computability to take a familiar form we assume that there is a $\kappa$-computable bijection between~$\kappa$ and the universe $H(\kappa)$. For $\kappa = \aleph_1$ this implies that all reals are constructible, but it is consistent with some subsets of~$\w_1$ not being constructible. For $\kappa = \aleph_2$ this is a consequence of some forcing axioms (for example PFA), which imply the failure of CH. 
	\item What can be said about more complicated groups? Some of the results can be extended to homogeneously completely decomposable groups (see~\cite{DowneyMelnikov:EffectivelyCategoricalGroups}). It may be interesting to investigate the effective properties of uncountable such groups. 
\end{enumerate}





\section{Preliminaries} 
\label{sec:preliminaries}

We start with a few basic facts, most of which are well-known. We provide some details for completeness, and also because our presentation reflects a more dynamic approach than appears in literature. This will make them more convenient to use in the arguments in the rest of the paper. 

Recall that throughout this paper, we assume that $V=L$. A general reference for torsion-free abelian groups is \cite{Fuchs:Volume1}. The fine-structure tools we use appear in~\cite{Jensen}.

\subsection{Detachment, freeness, and clubs} 
\label{sub:detachment_freeness_and_clubs}

All groups we will discuss are abelian and torsion-free. A group~$G$ is \emph{free abelian} if it has a basis: a subset~$B$ which is linearly independent ($\sum m_i b_i = 0$ implies each $m_i = 0$, where $m_i\in \Int$ and $b_i\in B$) and spans~$G$ (every element of~$G$ is of the form $\sum m_i b_i$ for some $m_i\in \Int$ and $b_i\in B$). We will omit the adjective ``abelian'' and just call these groups free. For any infinite cardinal~$\kappa$, the free group of size~$\kappa$ will be denoted by~$\Z^\kappa$. 

\begin{fact} \label{fact:subgroups_are_free}
	Any subgroup of a free group is free. 	
\end{fact}

The following is a key notion. 

\begin{definition} \label{def:detachment}
If $G$ is a group and $H\subseteq G$ is a subgroup, we say that $H$ \emph{detaches in} $G$ if $G = H \oplus K$ for some subgroup $K\subseteq G$. We write $H \div G$.	
\end{definition}

If~$G$ is free then $H\div G$ if and only if some basis of~$H$ can be extended to a basis of~$G$ if and only if every basis of~$H$ can be extended to a basis of~$G$.

\begin{fact} \label{fact:detachment_and_freeness_of_quotient}
	Suppose that~$G$ is free and that $H\subseteq G$ is a subgroup. Then~$H$ detaches in~$G$ if and only if the quotient group $G/H$ is free. 
\end{fact}

We also remark that if~$G$ is torsion-free abelian and $H$ is a subgroup of~$G$, then $H$ is a pure subgroup of~$G$ if and only if $G/H$ is torsion-free.

\smallskip

If $H$ is a subgroup of a group~$G$ then we write $[H,G]$ to denote the interval in the lattice of subgroups: it is the collection of all subgroups $K\subseteq G$ such that $H\subseteq K$.

\begin{proposition} \label{prop:detachment_is_forever}
	Suppose that $H$ detaches in $G$. Then~$H$ detaches in every subgroup $K\in [H,G]$. 
\end{proposition}

\begin{proof}
If $G$ is free then this follows from \cref{fact:subgroups_are_free} and 
\cref{fact:detachment_and_freeness_of_quotient}, but it also holds for 
arbitrary $G$.
Suppose that $G = H \oplus G'$ for some $G' \subseteq G$.  Let $K' = K \cap G'$.  Then $K = H \oplus K'$.  For if $g \in K$, then $g \in G$, so $g = h + k$ where $h \in H$ and $k\in G'$ and this decomposition is unique.  Since $h,g \in K$, we have $k \in K$, so $k \in K'$.
\end{proof}

\begin{proposition} \label{prop:Melnikov_done_short}
	There is a countable free group $G$ and a (by \cref{fact:subgroups_are_free}, free) pure subgroup~$H$ of~$G$ which does not detach in~$G$, but every finitely generated pure subgroup of~$H$ does detach in~$G$.
\end{proposition}

\begin{proof}
	Let $G = \Int^\omega$.  Take any torsion-free, non-free countable abelian group~$K$; fix an epimorphism from~$G$ onto~$K$; let $H$ be its kernel. Every finitely generated pure subgroup of~$H$ detaches in~$G$ because of Pontryagin's criterion \cref{prop:extension_of_finite_P-independece}.
\end{proof}

We also note that \cref{exmp:Melnikov} gives a direct construction of such groups~$H$ and~$G$. In the notation of that example, we may set $G = \Int^{\w+1}$ and let~$H$ be the span of~$U$.


\smallskip

A sequence $\seq{G_{\alpha}}_{\alpha<\gamma}$ of groups of some ordinal length~$\gamma$ is \emph{increasing} if $\alpha<\beta$ implies $G_{\alpha}\subseteq G_\beta$; it is \emph{continuous} if for all limit $\alpha<\gamma$, $G_\alpha = \bigcup_{\beta<\alpha} G_\beta$. A \emph{filtration} of a group~$G$ is a sequence $\bar G = \seq{G_{\alpha}}$ such that $\bar G$ is increasing and continuous, $G = \bigcup_{\alpha<\gamma} G_\alpha$, and $|G_\alpha| \le |\alpha|$ for all~$\alpha$. 

If~$\gamma$ is regular and~$G$ is a group of universe~$\gamma$ then all filtrations of~$G$ agree on a club; in fact, for club many~$\alpha$, $G_\alpha = G\cap \alpha$. We decide that the \emph{standard filtration} of a group~$G$ of universe a regular cardinal~$\gamma$ is defined by $G_\alpha = \Span(G\cap \alpha)$. 

\smallskip


\begin{definition} \label{def:detachment_set}
	Let $\bar G = \seq{G_{\alpha}}_{\alpha<\gamma}$ be increasing and continuous. The \emph{detachment set} of~$\bar G$ is
\[
	\Detach{\bar G} = \left\{ \alpha <\gamma \,:\,  \forall \beta \in (\alpha,\gamma) \,\,(G_\alpha \div G_\beta) \right\}.
\]
\end{definition}

If~$\gamma$ is regular and~$\bar G, \bar G'$ are two filtrations of a group of universe~$\gamma$, then $\Detach{\bar G}$ and $\Detach{\bar G'}$ agree on a club; this uses \cref{prop:detachment_is_forever}. In this case we will write $\Detach{G}$ for $\Detach{\bar G}$, where $\bar G$ is the standard filtration of~$G$. 

\smallskip

The following can essentially be found in \cite{Eklof1977c}; see \cite[IV.1.7]{EklofMekler}



\begin{proposition} \label{prop:club_and_freeness}
	Let $\gamma$ be a limit ordinal and let $\bar G = \seq{G_{\alpha}}_{\alpha<\gamma}$ be a filtration of a group~$G_\gamma$. Suppose that for all~$\alpha<\gamma$, $G_\alpha$ is free. 
	\begin{enumerate}
		\item If $\Detach{\bar G}$ contains a club of~$\gamma$ then~$G_\gamma$ is free.
		\item If~$\gamma$ is a regular cardinal and~$G_\gamma$ is free then $\Detach{\bar G}$ contains a club of~$\gamma$. 
	\end{enumerate}
\end{proposition}

\begin{proof}
	For (2), let~$B$ be a basis for~$G_\gamma$. There are club many~$\alpha<\gamma$ for which $G_\alpha = \Span(B\cap \alpha)$; each such~$\alpha$ belongs to $\Detach{\bar G}$.

	\smallskip

	For~(1), suppose that~$C\subseteq \Detach{\bar G}$ is closed and unbounded. We may assume that $G_0$ is the trival group and that $0\in C$. For $\alpha\in C$ let $\alpha' = \min C\setminus (\alpha+1)$  be the next element of~$C$ after~$\alpha$. Then $G_\alpha \div G_{\alpha'}$; choose some $H_\alpha$ such that $G_{\alpha'} = G_\alpha \oplus H_\alpha$. Then $G_\gamma = \bigoplus_{\alpha\in C} H_\alpha$. Each~$H_\alpha$ is free (as $G_{\alpha'}$ is free). If~$B_{\alpha}$ is a basis of~$H_\alpha$, then $\bigcup_{\alpha\in C} B_\alpha$ is a basis of~$G_\gamma$. 
\end{proof}

If~$\gamma$ is a regular cardinal and $\seq{G_{\alpha}}$ is a filtration of a group~$G$ of universe~$\gamma$, then the relation $G_\alpha\div G_\beta$ is $\gamma$-c.e.; we need to search for a complement for~$G_\alpha$ in~$G_\beta$ (when~$G_\beta$ is free, equivalently we search for a basis of $G_\beta/G_\alpha$).  We will see that for some~$\gamma$ this relation will be $\Sigma^0_1(L_\gamma)$-complete, but for other~$\gamma$ the relation will be $\gamma$-computable. Note that the standard filtration of~$G$ is $G$-computable. 

\begin{remark} \label{rmk:detachment_set_is_definable}
	Let~$\gamma$ be a limit ordinal; let $\bar G = \seq{G_{\alpha}}_{\alpha<\gamma}$ be a filtration of a group~$G_\gamma$. Suppose that~$\Detach{\bar G}$ contains a club of~$\gamma$. Then 
	\[
	\Detach{\bar G} = \left\{ \alpha <\gamma \,:\,  G_\alpha \div G_\gamma \right\}.
\] One direction follows from \cref{prop:detachment_is_forever}; the other from the proof of \cref{prop:club_and_freeness}.
\end{remark}


\subsection{$\Sigma^1_1$ completeness of finding clubs, and a class arising from the proof of square principles} 
\label{sub:long_name}

We saw that identifying a free group reduces to finding club subsets of the definable set $\Detach{G}$. Thus, our stated result would imply that existence of a club subset is $\Sigma^1_1$-complete. This is indeed the case; this was proved for $\kappa = \w_1$ by Fokina et al.\ in \cite{FokinaEtAl}. The proof generalises. We will need this fact and will need to get more information from its proof. Most material in this subsection can be found in \cite{Jensen}.

\smallskip

Here is a key notion. 

\begin{definition} \label{def:the_singular_place}
For a singular ordinal~$\alpha$, we let $\sing{\alpha}$ be the least ordinal $\beta\ge \alpha$ such that there is a cofinal sequence in~$\alpha$ of order-type smaller than~$\alpha$ which is definable over~$J_\beta$. 	
\end{definition}
In other words this is the first place at which we recognise that~$\alpha$ is singular. The sets~$J_\beta$ are Jensen's modification of the~$L_\alpha$ hierarchy which is required to make fine structure theory work (the sets~$J_\beta$ are closed under the rudimentary functions). The details are unimportant, and for sufficiently nice ordinals~$\alpha$ we have $L_\alpha = J_\alpha$ anyway. We will use some basic facts which hold for both hierarchies. For example, the function $\alpha\mapsto J_\alpha$ is $\Sigma_1$-definable in every~$J_\beta$ for $\beta>\alpha$. Also, the subsets of~$J_\beta$ which are elements of~$J_{\beta+1}$ are precisely the ones definable over~$J_\beta$ (with parameters). 

We note:
\begin{itemize}
	\item The function $\alpha\mapsto \sing{\alpha}$ is $\Sigma_1$-definable, and so its restriction to ordinals below a cardinal~$\kappa$ is partial $\kappa$-computable. 
\end{itemize}
The domain of this function, the set of singular ordinals below a cardinal~$\kappa$, may fail to be $\kappa$-computable; it is merely $\kappa$-c.e. Note that this only happens when $\kappa$ is a limit cardinal. If~$\kappa$ is a successor cardinal then the restriction of the function $\alpha\mapsto \sing{\alpha}$ to ordinals below~$\kappa$ is $\kappa$-computable. 



\begin{definition} \label{def:class_E}
	The class~$E$ consists of all the singular ordinals~$\alpha$ such that for some $\beta\in (\alpha, \sing{\alpha})$:
	\begin{itemize}
		\item $J_{\beta} \models \textrm{ZF}^-$;
		\item $\alpha$ is the greatest cardinal of $J_{\beta}$;
		\item for some $p\in J_{\beta}$, $J_{\beta}$ is the least (fully) elementary substructure $M\prec J_{\beta}$ such that $p\in M$ and $M\cap \alpha$ is transitive. 
	\end{itemize}
\end{definition}

Suppose that $\alpha\in E$ and let~$\beta>\alpha$ witness this fact. Then~$J_\beta$ can be presented as the countable union $\bigcup M_i$, with $M_0 = \{p\}$ and each~$M_{i+1}$ being the $\Sigma_i(J_\beta)$-Skolem hull of $M_i\cup \sup({M_i\cap \alpha})$. The sequence $\seq{M_i}$ is definable over $J_{\beta+1}$. However, for all~$i$, the process generating~$M_i$ is definable over~$J_\beta$. Since $\beta< \sing{\alpha}$, $M_i\cap \alpha$ is bounded below~$\alpha$. This implies that:
\begin{itemize}
	\item Each $\alpha\in E$ has countable cofinality, and $\sing{\alpha} = \beta+1$. 
\end{itemize}

The definition of~$E$ was designed to ensure the following:

\begin{lemma} \label{lem:getting_into_E}
	Let~$\kappa$ be regular and uncountable; let $q\in L_{\kappa^+}$. Let~$M$ be the least elementary substructure of~$L_{\kappa^+}$ such that $q\in M$ and $M\cap \kappa$ is transitive. Let $\pi\colon M\to J_\beta$ be the Mostowski collapse; let $\alpha = \pi(\kappa) = M\cap \kappa$. Then $\alpha\in E$, witnessed by~$\beta$. 
\end{lemma}

The main idea, for showing that $\beta<\sing{\alpha}$, is that if $\gamma<\alpha$ and $f\colon \gamma\to \alpha$ is $J_\beta$-definable and cofinal, then the same definition over~$M$ (equivalently $L_{\kappa^+}$) gives a cofinal $\hat f\colon \gamma\to \kappa$, which is impossible. It follows that if~$\kappa$ is regular, then $E\cap \kappa$ is stationary in $\kappa$: for any club~$C$ of~$\kappa$, consider the least elementary $M\prec L_{\kappa^+}$ such that $C\in M$ and $M\cap \kappa$ is transitive. A similar argument gives the $\Sigma^1_1$-completeness of containing a club. We will make use of the following tool. 

\begin{definition} \label{def:the_Sigma_11_approximating_set}
	Let~$\kappa$ be regular and uncountable, let $B\in 2^\kappa$, and let $\forall X\,\vphi$ be a $\Pi^1_1$ formula, where $\vphi$ is first-order with parameter $r\in L_\kappa$ .We let $F(B,\vphi)$ be the set of singular ordinals $\alpha<\kappa$ such that $r\in J_\alpha$, $B\rest \alpha\in J_{\sing{\alpha}}$ and for all $X\in J_{\sing{\alpha}}$, $(J_\alpha,B\rest \alpha,X)\models \vphi$. 
\end{definition}

That is, $\alpha\in F(B,\vphi)$ if we believe that the $\Pi^1_1$ property under discussion holds of~$(J_\alpha,B\rest \alpha)$, where we limit the second-order quantifiers to subsets of~$\alpha$ which are only constructed at stages at which we still think that~$\alpha$ is regular.

\begin{lemma} \label{lem:reflecting_Pi11_and_Sigma11_sets_on_F}
	Let~$\kappa$ be a regular cardinal, let $B\in 2^\kappa$, and let $\forall X\,\vphi$ be a $\Pi^1_1$ formula. 
		\begin{enumerate}
			\item If $(L_\kappa,B)\models \lnot\forall X\,\vphi$ then $F(B,\vphi)$ is nonstationary in~$\kappa$.
			\item If $(L_\kappa,B)\models \forall X\,\vphi$ then $E\cap F(B,\vphi)$ is stationary in~$\kappa$. 
		\end{enumerate}
\end{lemma}

\begin{proof}
	Let~$r$ be the parameter for~$\vphi$. 

	\smallskip

	For~(1), we build a continuous and increasing sequence $\seq{M_i}_{i<\kappa}$ of elementary submodels of~$L_{\kappa^+}$ such that $r,B\in M_0$ and $\alpha_i = \kappa\cap M_i$ is an element of~$\kappa$; the set $\{\alpha_i\,:\, i<\kappa\}$ is closed and unbounded in~$\kappa$ (we let $\alpha_i\in M_{i+1}$). Let $\pi_i\colon M_i \to J_{\beta_i}$ be the Mostowski collapse. The argument above shows that $\beta_i<\sing{\alpha_i}$. There is some $X\in M_0\cap 2^\kappa$ such that $(L_\kappa,B,X)\models \lnot \vphi$; then $X\rest{\alpha_i} = \pi_i(X)\in J_{\beta_i}$ (and $B\rest{\alpha_i}\in J_{\beta_i}$) and $(J_{\alpha_i},B\rest{\alpha_i},X\rest{\alpha_i})\models \vphi$ (as $J_{\beta_i}$ thinks it does, and this is absolute). Hence the club $\left\{ \alpha_i \,:\, i<\kappa   \right\}$ is disjoint from $F(B,\vphi)$. 

	\smallskip

	For~(2), let~$C$ be a club of~$\kappa$. Let $M\prec L_{\kappa^+}$ be least such that $r,C,B\in M$ and $M\cap \kappa \in \kappa$. Let $\pi\colon M\to J_\beta$ be the Mostowski collapse and let $\alpha = \pi(\kappa) = M\cap \kappa$. Then $\alpha\in E\cap C$, and $\sing{\alpha} = \beta+1$. Suppose that $X\in J_{\beta+1}\cap 2^\alpha$. It is definable over~$J_\beta$, say with parameter~$q$. Let~$\hat X$ be the interpretation of the same definition over~$M$ (equivalently $L_{\kappa^+}$), with parameter $\pi^{-1}(q)$. Then $(L_\kappa,B,\hat X)\models \vphi$. It follows that $(J_\alpha,B\rest{\alpha},X)\models \vphi$, so $\alpha\in F(B,\vphi)$.\footnote{In the definition of~$F(B,\vphi)$ we could replace $\sing{\alpha}$ by $\sing{\alpha}-1$, assuming that we are restricting ourselves to $\alpha\in E$.} 
\end{proof}

\begin{corollary} \label{cor:Sigma11_completeness_of_containing_a_club}
	Let~$\kappa$ be a successor cardinal. The nonstationary ideal on~$\kappa$ (equivalently the club filter on~$\kappa$) is $\Sigma^1_1$-complete. In fact, the restriction of the nonstationary ideal to~$E\cap \kappa$ is $\Sigma^1_1$-complete. That is, for any $\Sigma^1_1(L_\kappa)$ set $A\subseteq 2^\kappa$ there is a $\kappa$-computable function $f\colon 2^\kappa\to 2^\kappa$ such that for all~$Y\in 2^\kappa$, $f(Y)\subseteq E$, and $Y\in A$ if and only if $f(Y)$ is nonstationary. 
\end{corollary}

\begin{proof}
	Let $\exists X\,\vphi$ be the formula defining~$A$; we let $f(Y) = E\cap F(Y,\lnot\vphi)$. Recall that the set of singular ordinals below~$\kappa$ is $\kappa$-computable; this implies that~$E\cap \kappa$ is $\kappa$-computable and that $F(Y,\lnot\vphi)$ is~$Y$-computable, uniformly in~$Y$. 
\end{proof}

A key fact that we will use for $\kappa\ge \aleph_2$ is the following, which is \cite[Thm.5.1]{Jensen}.

\begin{theorem}[Jensen] \label{thm:square_really}
	The class~$E$ does not reflect at any singular ordinal. That is, if $\alpha$ is singular then $E\cap \alpha$ is nonstationary in~$\alpha$. 
\end{theorem}

The proof of this fact is complicated. It is part of the proof of the square principle in~$L$.

\subsection{Twisting a group} 
\label{sub:twisting_a_group}

The plan for proving \cref{thm:identifying_free_groups} for the case of successor cardinals is to take a set~$Y\subseteq \kappa$ and produce a $Y$-computable group~$G$ such that $\Detach{G} = \kappa\setminus f(Y)$, where~$f$ is given by \cref{cor:Sigma11_completeness_of_containing_a_club}. A main tool would be to take a group~$G_\alpha$ which we have already constructed, and ensure that it does not detach in~$G$ by ensuring that it does not detach in~$G_{\alpha+1}$. On the other hand we need to ensure that for all~$\beta<\alpha$, if we already declared that we want $G_\beta$ to detach in~$G$, then $G_\beta$ detaches in~$G_{\alpha+1}$. We need to ``twist'' $G_\alpha$ without further twisting any $G_\beta$ for $\beta<\alpha$. 

The idea is to use \cref{prop:Melnikov_done_short}. We generalise it to possibly uncountable groups by picking out countable pieces. 

\begin{proposition} \label{prop:twisting_lemma}
	Suppose that $\seq{H_n}$ is an increasing sequence of free groups such that for all~$n$, $H_{n}\div H_{n+1}$; so $H_\w = \bigcup_n H_n$ is free as well. There is a free group~$G$ extending~$H_\w$ (with $|G| = |H_\w|$) such that $H_\w \ndiv G$ but for all~$n$, $H_n\div G$. The group~$G$ can be obtained effectively from the sequence $\seq{H_n}$. 
\end{proposition}

We write $\twist(\seq{H_n})$ for the group~$G$. 

\begin{proof}
	Without loss of generality we assume that $H_0$ is the trivial group. As in the proof of \cref{prop:club_and_freeness}, choose subgroups~$K_n$ such that $H_{n+1} = H_n\oplus K_n$, so $H_\w = \bigoplus_n K_n$. As each~$K_n$ is free, we write $K_n = P_n\oplus Q_n$, where $Q_n\cong \Z$. Let $P = \bigoplus_n P_n$ and $Q = \bigoplus Q_n$. 

	Using \cref{prop:Melnikov_done_short} we can find a countable free group $R\supseteq Q$, such that $Q\ndiv R$, but for any~$n$, $Q_{<n} = \bigoplus_{m<n} Q_m$ does detach inside~$R$. We let $G = P\oplus R$. 

	It follows that for all~$n$, $P\oplus  Q_{<n}$ detaches in~$G$. As $H_n$ detaches in $P\oplus Q_{<n}$, and detachment is transitive, we see that each~$H_n$ detaches in~$G$. 

	It also follows that $H_\w = P\oplus Q$ does not detach in~$G$; if $H_\w\div G$ then $Q\div G$ and as $Q\subseteq R\subseteq G$ we would have $Q\div R$ (\cref{prop:detachment_is_forever}).
\end{proof}


\section{Identifying free groups} 
\label{sec:identifying_free_groups}

We now turn to the proof of \cref{thm:identifying_free_groups}.

\subsection{The successor case} 
\label{sub:the_successor_case}

We prove \cref{thm:identifying_free_groups} where~$\kappa$ is a successor cardinal:

\begin{theorem} \label{thm:identifying_freeness-successor}
	Let~$\kappa$ be a successor cardinal. The set of free abelian groups of universe~$\kappa$ is $\Sigma^1_1(L_\kappa)$-complete. 
\end{theorem}

\begin{proof}
	Let~$A\subseteq 2^\kappa$ be $\Sigma^1_1$. Given $Y\in 2^\kappa$ we (uniformly) compute an abelian group~$G_Y$ which is free if and only if~$Y\in A$. To begin, we find a set $U = U_Y\subseteq E\cap \kappa$, effectively obtained from~$Y$, such that $Y\in A$ if and only if~$U$ is nonstationary (\cref{cor:Sigma11_completeness_of_containing_a_club}). Without loss of generality, $U\cap \kappa^- = \emptyset$, where $\kappa^-$ is the cardinal predecessor of~$\kappa$. The point here is that  every ordinal in $(\kappa^-,\kappa)$ is singular, and so~$E$ does not reflect at any ordinal in this interval (\cref{thm:square_really}).

	 We will build a filtration $\seq{G_\alpha}_{\alpha<\kappa}$ of a group~$G_Y$ such that each~$G_\alpha$ is free, and $\Detach{\seq{G_\alpha}} = \kappa \setminus U$; our desired equivalence then holds by \cref{prop:club_and_freeness}. 

	We define the sequence $\seq{G_\alpha}$ computably in~$U$. The construction is of course by recursion on~$\alpha$.  In order for the construction to proceed as we will shortly describe, we will need to maintain the following:
	\begin{enumerate}
		\item[(i)] each $G_\alpha$ is free; and
		\item[(ii)] for all $\beta<\alpha$, if $\beta\notin U$ then $G_\beta \div G_\alpha$. 
	\end{enumerate}

	We start with $G_{0}$ being the trivial group. Taking $\alpha<\kappa$, we assume that~$G_\beta$ is defined for all $\beta<\alpha$ and that the above hypotheses hold below~$\alpha$. 

	\medskip

	\noindent\textit{Case 1: $\alpha$ is a successor ordinal and $\alpha-1\notin U$.} We let $G_\alpha = G_{\alpha-1}\oplus \Z$. (i) for~$\alpha$ holds easily. (ii) holds because $G_{\alpha-1}\div G_\alpha$ and detachment is transitive. 

	\smallskip

	\noindent\textit{Case 2: $\alpha$ is a successor ordinal and $\alpha-1\in U$.} Since $\alpha-1\in E$ we know that $\cf(\alpha-1)= \w$. We can choose an increasing and cofinal sequence $\seq{\alpha_n}$ in~$\alpha-1$ which is disjoint from~$U$, for example consisting of successor ordinals. By induction, for all~$n$, $G_{\alpha_n} \div G_{\alpha_{n+1}}$. We can thus apply \cref{prop:twisting_lemma}: we let $G_\alpha = \twist(\seq{G_{\alpha_n}})$. (i) holds by construction. For (ii), let $\beta<\alpha$, $\beta\notin U$. There is some~$n$ such that $\beta<\alpha_n$. By induction, $G_\beta\div G_{\alpha_n}$; by construction, $G_{\alpha_n}\div G_\alpha$.

	\smallskip

	\noindent\textit{Case 3: $\alpha$ is a limit ordinal.} We let $G_\alpha = \bigcup_{\beta<\alpha} G_\beta$. To verify (i) and~(ii) in this case we use the fact that there is a club~$C$ of~$\alpha$ which is disjoint from~$U$ (\cref{thm:square_really}). It follows that $\Detach{\seq{G_\beta}_{\beta<\alpha}}$ contains a club, and so $G_\alpha$ is free (\cref{prop:club_and_freeness}). (ii) follows from~\cref{rmk:detachment_set_is_definable}. 

	\medskip

	Note that in case 2, to perform the twist, we need a basis for~$G_\alpha$. However we know that~$G_\alpha$ is free, so we can simply search for a basis until we find it. Identifying that~$B$ is a basis of a free group~$G$ is $\kappa$-computable.

	\smallskip

	(ii) above implies that $\kappa\setminus U \subseteq \Detach{G_Y}$. However if $\alpha\in U$ then we ensured that $G_{\alpha} \ndiv G_{\alpha+1}$, so $\alpha\notin \Detach{G_Y}$. This completes the proof. 
\end{proof}

We remark that a non-effective, static construction of a $\aleph_1$-free group of size $\aleph_1$ with a prescribed detachment set can be found in \cite[IV]{EklofMekler}.


\subsection{Lightface weak compactness} 
\label{sub:lightface_weak_compactness}

A cardinal~$\kappa$ is weakly compact if and only if it is $\mathbf{\Pi}^1_1$-indescribable. Under $V=L$, for $B\in 2^\kappa$, say that~$\kappa$ is $\Pi^1_1(B)$-describable if there is a $\Pi^1_1(B)$ fact which holds for~$L_\kappa$ but not for $L_\lambda$ for any $\lambda<\kappa$ (we may restrict ourselves to regular $\lambda<\kappa$, since this is expressible by a $\Pi^1_1$-statement). That is, if for some first-order~$\vphi$, for all~$X\in 2^\kappa$, $(L_\kappa,B,X)\models \vphi$, but for all $\lambda<\kappa$, for some $X\in 2^\lambda$, $(L_\lambda, B\rest \lambda, X)\models \lnot\vphi$. For example, the least inaccessible cardinal is $\Pi^1_1$-describable. A cardinal~$\kappa$ is weakly compact if and only if it is $\Pi^1_1(B)$-indescribable for all $B\in 2^\kappa$. The next part of \cref{thm:identifying_free_groups} follows from the following theorem. 

\begin{theorem} \label{thm:identifying_inaccessible_free_groups}
	Let~$\kappa$ be an inaccessible cardinal, and let $B\in 2^\kappa$. If~$\kappa$ is $\Pi^1_1(B)$-describable then the collection of free abelian groups of size~$\kappa$ is $\Sigma^1_1(B)$-complete. 
\end{theorem}

Note that if $\kappa$ is $\Pi^1_1(B)$-describable and $C$ $\kappa$-computes~$B$, then $\kappa$ is also $\Pi^1_1(C)$-describable. Also note that \cref{thm:identifying_inaccessible_free_groups} implies that the collection of free abelian groups on the least inaccessible cardinal is $\Sigma^1_1$-complete. 


\subsection{An elaboration on square for inaccessible cardinals} 
\label{sub:elaborations_on_square_for_inaccessible_cardinals}

Toward proving \cref{thm:identifying_inaccessible_free_groups}, we need an elaboration on the class~$E$ above and on \cref{cor:Sigma11_completeness_of_containing_a_club}. Consider what would go wrong if we tried to replicate the proof of \cref{thm:identifying_freeness-successor} for an inaccessible cardinal~$\kappa$. One problem is that $E\cap \kappa$ is no longer~$\kappa$-computable, merely $\kappa$-c.e.; we will need to address this problem in the construction below by approximating the final filtration~$\seq{G_\alpha}$ while still building a computable group. A more serious obstacle is that~$E$ does reflect at all regular cardinals, and so unboundedly below~$\kappa$. This would mean that we will not be able to ensure that all the groups~$G_\alpha$ that we build along the way are free. We need to restrict ourselves to a sparser class which will be stationary in~$\kappa$ but not reflect (will not be stationary in any $\alpha<\kappa$). 

If~$\kappa$ is weakly compact then every stationary subset of~$\kappa$ reflects (being stationary is $\Pi^1_1$). Hence there is no hope to perform this construction in this case. And indeed, below we use this very fact to give an easy characterisation of free groups of a weakly compact size. Jensen showed that in~$L$, this is the only problematic case.

\smallskip

Recall \cref{def:the_Sigma_11_approximating_set} of the set $F(B,\vphi)$. 



\begin{proposition} \label{prop:F_does_not_reflect}
	Let~$\kappa$ be inaccessible and $\Pi^1_1(B)$-describable, say by the formula $\forall X\,\vphi$. Then $E\cap F(B,\vphi)$ is stationary in~$\kappa$, but does not reflect at any limit ordinal $\alpha<\kappa$. 
\end{proposition}

\begin{proof}
	That~$E\cap F(B,\vphi)$ is stationary follows from \cref{lem:reflecting_Pi11_and_Sigma11_sets_on_F}(2). 

	\smallskip

	Let $\alpha<\kappa$ be a limit ordinal. If~$\alpha$ is singular, then we know that~$E$ does not reflect at~$\alpha$. If~$\alpha$ is a regular cardinal then by assumption, $(L_\alpha,B\rest \alpha)\models \exists X\lnot \vphi$; by \cref{lem:reflecting_Pi11_and_Sigma11_sets_on_F}, $F(B,\vphi)$ is nonstationary in~$\alpha$. 
\end{proof}



For brevity let $F = E\cap F(B,\vphi)$. Replacing~$E$ by~$F$ in the proof of \cref{lem:reflecting_Pi11_and_Sigma11_sets_on_F} shows that the nonstationary ideal on~$\kappa$, in fact its restriction to~$F$, is $\Sigma^1_1(F)$-complete. Copying the construction proving \cref{thm:identifying_freeness-successor} shows that the collection of free abelian groups of size~$\kappa$ is $\Sigma^1_1(F)$-complete. However, this does not quite give \cref{thm:identifying_inaccessible_free_groups}, because~$F$ may fail to be $B$-computable; it is merely $B$-c.e. As mentioned above, we modify the construction to approximate the desired result.

\begin{proof}[Proof of \cref{thm:identifying_inaccessible_free_groups}]
Let~$F = E\cap F(B,\vphi)$, where $\kappa$ is $\Pi^1_1(B)$-describable, witnessed by the formula $\forall X\,\vphi$. Let~$A$ be $\Sigma^1_1(B)$, say defined by the formula $\exists X\,\psi$. Given $Y\in 2^\kappa$ we let $U = U_Y = F\cap F(Y,\lnot\psi)$. If $Y\in A$ then $F(Y,\lnot\psi)$ is nonstationary in~$\kappa$ (\cref{lem:reflecting_Pi11_and_Sigma11_sets_on_F}), and so~$U$ is nonstationary. If $Y\notin A$ then by the same lemma,~$U$ is stationary, as it contains $E\cap F((B,Y),\vphi\wedge \lnot\psi)$. 

Again our aim is to build a group~$G_Y$ of universe~$\kappa$ and a filtration $\bar G = \seq{G_\alpha}_{\alpha<\kappa}$ of~$G$ such that $\Detach{\bar G} = \kappa\setminus U$. The group~$G_Y$ needs to be $Y\oplus B$-computable, uniformly in~$Y$; but as mentioned above, the filtration $\bar G$ will not. 

What we do have, effectively from~$Y\oplus B$, is an enumeration of~$U$: an increasing and continuous sequence $\seq{U_s}_{s<\kappa}$ of sets in~$L_\kappa$ such that $U = \bigcup_{s<\kappa} U_s$. Namely we let~$U_s$ be the collection of $\alpha\in U$ such that $\sing{\alpha}< s$. Again the point is that the set of singular ordinals below~$\kappa$ is~$\kappa$-c.e.; once we see that~$\alpha$ is singular we can effectively, from $B\oplus Y$, check whether $\alpha\in U$ or not. Note that this means that for any cardinal~$\lambda<\kappa$, $U_\lambda = U\cap \lambda$. For all~$s$, $U_s\subseteq s$. 

We do define~$G_Y$ by building a $Y\oplus B$-computable increasing sequence $\seq{H_s}$. The problem with copying the previous construction is that at a late stage~$s$ we may see some relatively small~$\alpha$ enter~$s$. Now we could twist~$H_\alpha$ inside~$H_{s+1}$. But this would cause all the groups~$H_\beta$ for $\beta\in (\alpha,s]$ to be twisted inside~$H_{s+1}$ as well. This would result in our twisting at places outside~$U$ (and outside~$F$). At the very end this shouldn't matter; we could argue that the differences are washed outside some club of~$\kappa$. The difficulty though is to explain why \emph{each} group~$H_\gamma$ is free. Na\"{i}vely, if~$t$ is a limit of such stages~$s$ as above, then while~$U\cap t$ is nonstationary in~$t$, it is conceivable that the added twisting would cause a stationary amount of twisting, and then $H_{<t}$ would fail to be free. This in fact does not happen, but we prefer to present a modified construction. Our approach is to re-index the filtration. Namely, at stage~$s$ we declare that all the groups $G_\beta$ for $\beta\in (\alpha,s]$ are swallowed inside the new $G_{\alpha+1}$. 

So together with the sequence~$\seq{H_s}$ we define filtrations $\bar G_s = \seq{G_{\alpha,s}}_{\alpha<s}$ of~$H_s$ whose limit will be the desired filtration~$\seq{G_\alpha}$. So $H_s = G_{s,s}$. The inductive assumption that makes everything work is:
\begin{enumerate}
	\item[(i)] each~$H_s$ is free;
	\item[(ii)] $\Detach{\bar G_s} = s \setminus U_s$.  
	\item[(iii)] For all $t<s$ and all $\beta<t$, if $U_s\rest \beta = U_t\rest \beta$ then $G_{\beta,t} = G_{\beta,s}$. 
\end{enumerate}

Note that since $U_s\subseteq F$, (ii) implies that for all limit $s<\kappa$, $\Detach{\bar G_s}$ contains a club of~$s$. Suppose that these objects have been defined for all~$t<s$. At stage~$s$ we act as follows.

	\medskip

	\noindent\textit{Case 1: $s$ is a successor ordinal.} If $U_{s}= U_{s-1}$ then we let $H_{s} = H_{s-1}\oplus \Z$, $G_{\alpha,s} = G_{\alpha,s-1}$ for all~$\alpha<s$ and $G_{s,s} = H_{s}$. In this case ensuring that (i), (ii) and (iii) above hold for~$s$ is immediate. 

	Suppose that $U_s \ne U_{s-1}$; let~$\alpha$ be the least element of $U_s\setminus U_{s-1}$. For all $\beta\le \alpha$, we let $G_{\beta,s} = G_{\beta,s-1}$. By induction, $G_{\alpha,s}\div H_{s-1}$. Write $H_{s-1}$ as the direct sum $G_{\alpha,s}\oplus K_s$. Find a cofinal $\w$-sequence $\seq{\alpha_n}$ in~$\alpha$ disjoint from~$U_{s}$. We let $G_{\alpha+1,s} = \twist(\seq{G_{\alpha_n,s}})\oplus K_s$. So $G_{\alpha+1,s}\supset H_{s-1}$, and $G_{\alpha,s}$ does not detach in $G_{\alpha+1,s}$; but for all $\beta\in \alpha\setminus U_{s-1} = \alpha\setminus U_s$, $G_{\beta,s}\div G_{\alpha+1,s}$. 

	We then go on defining $G_{\beta,s}$ for $\beta\in (\alpha+1,s]$ as in the previous construction, twisting on elements of~$U_s$ and adding copies of~$\Int$ outside~$U_s$, taking unions at limit levels; we let $H_s = G_{s,s}$.
	The verification of~(i) and~(ii) proceeds as in the proof of \cref{thm:identifying_freeness-successor}, again using the fact that~$U_s$ does not reflect at any $\beta\le s$. (iii) holds by our instructions. 

	\smallskip

	\noindent\textit{Case 2: $s$ is a limit ordinal.} Let 
	\[
		\gamma = \sup \left\{ \beta<s \,:\,  \exists t<s\,\, \big(U_t\rest \beta = U_s\rest \beta\big) \right\}.
	\]
	For all $\beta < \gamma$ we let $G_{\beta,s} = \lim_{t\to s} G_{\beta,t}$; the limit exists by (iii), and this definition ensures that (iii) holds at~$s$ as well. Further, we claim that $H_{<s} = \bigcup_{t<s} H_t$ actually equals $G_{\gamma,s} = \bigcup_{\beta<\gamma} G_{\beta,s}$. This is because for each $t<s$ there is some $r\in (t,s)$ and some $\alpha<\gamma$ which enters~$U$ at stage~$r$; at stage~$r$ we define $G_{\alpha+1,r}$ to extend~$H_t$. Now by induction, for all $\beta<\gamma$, $G_{\beta, s}$ is free. We show that (ii) holds: $\Detach{\seq{G_{\beta,t}}_{\beta<\gamma}}$ equals $\gamma\setminus U_s$; this uses the fact that $U_s = \bigcup_{t<s} U_t$. For if $\beta\in U_t$ for some $t<s$ then the construction ensures that for all $r\in [t,s)$, $G_{\beta,r}\ndiv G_{\beta+1, r}$. And if $\beta\notin U_s$ then for all $t\in (\beta,s)$, for all $\alpha\in (\beta, t)$, $G_{\beta,t}\div G_{\alpha,t}$; for each~$\alpha\in (\beta,s)$ we can find some $t\in (\beta, s)$ such that $G_{\alpha,t} = G_{\alpha,s}$ and $G_{\beta,t} = G_{\beta,s}$. 

	Finally, the fact that~$U_s$ does not reflect at~$s$ implies that $G_{\gamma,s}$ is free. Now as at the successor case, we continue building the sequence $\seq{G_{\beta,s}}$ for $\beta\in (\gamma,s)$ (if $\gamma<s$) as in the proof of \cref{thm:identifying_freeness-successor}, using~$U_s$, and let~$H_s = G_{s,s}$. 

	\medskip

	This completes the construction; applying the argument above to $s = \kappa$ completes the proof. Also note that for all $\alpha< \kappa$, $|G_\alpha|\le |\alpha|$, as for each cardinal~$\lambda<\kappa$, $U_\lambda = U\rest\lambda$; this implies that for all $\beta<\lambda$, $G_\beta = G_{\beta,\lambda}$. 
\end{proof}


\subsection{The weakly compact case} 
\label{sub:the_weakly_compact_case}

We turn to the proof of \cref{thm:identifying_free_groups}(2). Unlike the previous cases, here we not only have to prove hardness, but also membership in the class. This membership follows from an easy characterisation. The boldface version of the following proposition (which applies to weakly compact cardinals) was observed by A.~Mekler in his Ph.D.\ thesis. 

\begin{proposition} \label{prop:weak_compact_characterisation_of_freeness}
	Let~$\kappa$ be an inaccessible cardinal, let $G$ be a group of universe~$\kappa$, and suppose that~$\kappa$ is $\Pi^1_1(G)$-indescribable. Then~$G$ is free if and only if every subgroup of~$G$ of size smaller than~$\kappa$ is free. 
\end{proposition}

Note that under our assumption that $V=L$, every subgroup of such a group~$G$ of size less than~$\kappa$ is an element of~$L_\kappa$ (we say that it is \emph{$\kappa$-finite}). The collection of $\kappa$-finite free groups is $\kappa$-c.e.\ (as usual, search for a basis; every basis is $\kappa$-finite). This shows that for any~$B\in 2^\kappa$, if $\kappa$ is $\Pi^1_1(B)$-indescribable then the index-set of the $B$-computable free groups is $\Pi^0_2(B)$. If~$\kappa$ is weakly compact, this shows that the collection of all free abelian groups of size~$\kappa$ is $\Pi^0_2(L_\kappa)$.

\begin{proof}[Proof of \cref{prop:weak_compact_characterisation_of_freeness}]
	Recall that the standard filtration of~$G$ is defined by letting $G_\alpha = \Span(G\cap \alpha)$, and that we let~$\Detach{G}$ be the detachment set given by this standard filtration. 

	Let $\lambda \le \kappa$ be regular. First note that if $\lambda$ is closed under the group operation ($G_\lambda = G\rest \lambda$) then for all $\alpha<\lambda$, $G_\alpha\in L_\lambda$. 

	Consider the $\Sigma^1_1$ sentence~$\psi$ which for such $\lambda \le \kappa$, says that:
	\begin{itemize}
		\item for all $\alpha<\lambda$, $G_\alpha$ is free (has a basis in $L_\lambda$); and
		\item $\Detach{G_\lambda}$ contains a club. 
	\end{itemize}
 	For such~$\lambda$,  $(L_\lambda, G_\lambda)\models \psi$ if and only if~$G_\lambda$ is free. By indescribability, if~$G$ is not free then there is some regular $\lambda<\kappa$ such that $G_\lambda = G\rest \lambda$ and $(L_\lambda,G\rest\lambda)\models \lnot\psi$. 
\end{proof}

Assuming that $\kappa$ is weakly compact, as observed, this implies that the index set of the computable free abelian groups on~$\kappa$ is~$\Pi^0_2$. However above the halting problem we can save a quantifier. 

\begin{proposition} \label{prop:weakly_compact:free_is_Pi01}
	Let~$\kappa$ be weakly compact. Then the collection of free abelian groups on~$\kappa$ is $\Pi^0_1(\emptyset')$-complete in the collection of groups.
\end{proposition}

\begin{proof}
	First we show that freeness is indeed $\Pi^0_1(\emptyset')$. The point is that if $\lambda<\kappa$ is a cardinal and $H\in L_\lambda$ is a subgroup of~$G$, then~$H$ is free if and only if~$H$ has a basis in~$L_\lambda$. So \cref{prop:weak_compact_characterisation_of_freeness} implies that~$G$ is free if and only if for all cardinals~$\lambda$, $L_\lambda$ sees that every $\lambda$-finite subgroup of~$G$ is free. The set of cardinals is $\kappa$-computable from (indeed $\kappa$-equi-computable with) the complete $\Sigma^0_1(L_\kappa)$ set~$\emptyset'$. 

	\smallskip

	For completeness, we first observe that the collection of free abelian groups on~$\kappa$ is $\Delta^0_1$-hard; this only requires fixing two groups, one free and one not. Now let $A\subseteq 2^\kappa$ be $\Pi^0_1$; say $Y\in A$ if and only if $(L_\kappa,Y)\models \forall \alpha\,\,\psi(\alpha)$, for some formula~$\psi$ with bounded quantifiers. Then uniformly in~$Y$ we build groups~$G_\alpha$, for $\alpha<\kappa$, such that~$G_\alpha$ is free if and only if $(L_\kappa,Y)\models \psi(\alpha)$; and let~$G = \bigoplus G_\alpha$. This construction of course relativises to any oracle. 
\end{proof}

The following completes the proof of \cref{thm:identifying_free_groups}.

\begin{proposition} \label{prop:weakly_compact:Pi02_complete}
	Suppose that $\kappa$ is inaccessible and $\Pi^1_1$-indescribable. Then the index-set of the computable free abelian groups on~$\kappa$ is $\Pi^0_2$-complete. 
\end{proposition}

\begin{proof}
	We have already observed that it is $\Pi^0_2$. We prove hardness. The argument for \cref{prop:weakly_compact:free_is_Pi01} shows that it is sufficient to prove $\Sigma^0_1$-hardness. 

	\smallskip

	Let~$A$ be a $\kappa$-c.e.\ set; we describe a procedure yielding, given $\alpha<\kappa$, a $\kappa$-computable group~$G = G(\alpha)$ such that $G(\alpha)$ is free if and only if $\alpha\in A$. 

	The idea is to follow the construction of the proof of theorem \cref{thm:identifying_freeness-successor} up to the next cardinal~$\alpha^+$ (the least cardinal~$\lambda$ such that $\alpha<\lambda$). We twist along~$E$ (\cref{def:class_E}) as long as we don't see~$\alpha$ enter~$A$. The point is that $\alpha \in A$ if and only if $L_{\alpha^+} \models \alpha\in A$, and that~$E$ is stationary in~$\alpha^+$ but not between~$\alpha$ and~$\alpha^+$. So $\alpha\in A$ if and only if and only if at some point below~$\alpha^+$ we stop twisting altogether. Once we get to~$\alpha^+$ we cannot continue the construction. Of course, effectively, we don't know that we reached~$\alpha^+$, so we keep waiting to tell whether it is in~$E$ or not; to prevent us from producing a partial group, on the side we keep building a copy of $\Z^\kappa$ to add to our group. 

	\smallskip

	Here are the details more formally. Fix a $\kappa$-effective enumeration~$\seq{A_s}$ of~$A$; $A_s$ is the set of $x<s$ such that $J_s$ sees that $x\in A$. For any cardinal~$\lambda$, $L_\lambda \prec_{\Sigma_1} L_\kappa$, so for any cardinal $\lambda$, $A_\lambda = A\cap \lambda$. 

	Fix $\alpha<\kappa$. Computably we build an increasing and continuous sequence of groups $\seq{H_{\beta}}_{\beta\in [\alpha,\alpha^+]}$ and a continuous and non-decreasing function $f\colon [\alpha,\kappa)\to \alpha^+ +1$. We then let $G_s =  H_{f(s)}\oplus \Z^s$ for all $s\in [\alpha,\kappa]$. This is done so that the sequence $\seq{G_s}$ is increasing, continuous and~$\kappa$-computable, so $G = G_\kappa$ is a $\kappa$-computable group. At every stage we increase~$f$ by at most one, so for all $t\in [\alpha,\kappa]$, the range of $f\rest{[\alpha,t)}$ is an initial segment of $[\alpha,\alpha^+]$; so to define the groups~$H_\beta$ we define the group~$H_{f(t)}$ whenever we increase~$f$. 

	We start with $H_\alpha$ being the trivial group, and $f(\alpha)=\alpha$. Now let $t\in (\alpha,\kappa]$, and suppose that $f(s)$ and $H_{f(s)}$ have been defined for all $s\in [\alpha,t)$. Now there are several options. 

	\medskip

	\noindent\textit{Case 1: $t$ is a limit ordinal.} We let $f(t) = \sup_{s\in [\alpha,t)} f(s)$. If~$f$ is constant on a final segment of~$t$ then~$H_{f(t)}$ is already defined. Otherwise we let $H_{f(t)} = \bigcup_{s\in [\alpha,t)} H_{f(s)}$. 

	\smallskip

	In the other cases, $t$ is a successor ordinal; let $\beta = f(t-1)$. 

	\smallskip

	\noindent\textit{Case 2: $\beta$ is a successor ordinal.} We let $f(t) = \beta+1$ and $H_{\beta+1} = H_\beta \oplus \Z$.
	
	\smallskip

	\noindent\textit{Case 3: $\beta$ is a limit ordinal and $t<\sing{\beta}$.} (Of course this includes the case that~$\beta$ is a regular cardinal, which will be $\alpha^+$). We let $f(t) = \beta$.
	
	\smallskip

	\noindent\textit{Case 4: $\beta$ is a limit ordinal and $t=\sing{\beta}$.} We let $f(t) = \beta+1$. In this case, by induction, $H_\beta$ is free; we search for a basis and find it. Also by induction, $\Detach{\seq{{H_{\gamma}}}_{\gamma\in [\alpha,\beta)}}$ contains $\gamma\setminus E$, and so contains a club. 

	If $\beta\in E$ and $\alpha\notin A_t$ then we twist: we find a sequence $\seq{\beta_n}$ cofinal in~$\beta$ and disjoint from~$E$, and let $H_{\beta+1} = \twist({\seq{H_{\beta_n}}})$. 

	If $\beta\notin E$, or $\alpha\in A_t$, we let $H_{\beta+1} = H_\beta \oplus \Z$. 
	
	\smallskip

	This concludes the construction. By induction we can see that $\range f= [\alpha,\alpha^+]$. By induction we see that for all $\beta\in [\alpha,\alpha^+)$, $H_\beta$ is free, and that if $\alpha\in A$ then $\Detach{\seq{H_\beta}}$ contains a final segment of~$\alpha^+$, and otherwise equals $[\alpha,\alpha^+)\setminus E$, which does not contain a club. Hence $\alpha\in A$ if and only if~$H_{\alpha^+}$ is free if and only if~$G$ is free. 
\end{proof}

\begin{remark} \label{rmk:low_oracles}
	What about $\Pi^0_2(B)$-completeness for oracles~$B$ which do not compute~$\emptyset'$? We do not know much, but we can show that if~$B$ is low and~$\kappa$ is inaccessible and $\Pi^1_1$-indescribable then  the index set of the $B$-computable free abelian groups is $\Pi^0_2(B)$-complete in a strong sense: there is a $\kappa$-computable (not merely computable in~$B$) function~$f$ which reduces the complete $\Pi^0_2(B)$-set to the set of $\kappa$-computable (not just~$B$-computable) free abelian groups. 

	As above it suffices to prove $\Sigma^0_1(B)$-hardness. We sketch the argument. Let~$\alpha<\kappa$; we effectively build a $\kappa$-computable group~$G$ which is free if and only if $\alpha\in B'$. Fix a $\kappa$-computable approximation $\seq{B'_s}_{s<\kappa}$ for~$B'$. 

	We combine ingredients from previous constructions. We define an increasing and continuous sequence $\seq{H_\beta}_{\beta\ge \alpha}$, and along with it an approximation to a filtration~$\seq{G_{\beta}}_{\beta\ge \alpha}$ as in the proof of \cref{thm:identifying_inaccessible_free_groups}. As this is a sketch we ignore this approximation and discuss the final result $\seq{G_\beta}$. This is done so that $G_\beta$ is twisted inside~$G_{\beta+1}$ if and only if $\beta\in E$, we see that~$G_\beta$ is free, and $\alpha\notin B'_s$ for all~$s$ in some final segment of~$\beta$. As long as these conditions do not hold we keep ``puffing up'' the group with copies of~$\Int$ so that in the end we do get a $\kappa$-computable group; see for example the proof of \cref{thm:coding_zero_double} below. 

	Why does this work? Suppose first that $\alpha\notin B'$. We will show that for some~$\lambda<\kappa$, $G_\lambda$ is not free. Fix some~$\beta$ such that $\alpha\notin B'_s$ for all $s\ge \beta$, and assume that~$G_\beta$ is free. Then we twist at every $\gamma\in [\beta,\beta^+) \cap E$, which shows that $G_{\beta^+}$ is not free. 

	Suppose that $\alpha\in B'$. Since $\beta\in B'_s$ for all~$s$ in a final segment of~$\kappa$, eventually we stop twisting; we just need to show that each~$G_\beta$ is free, that is, the construction does not die prematurely. Since we only twist along~$E$, the first non-free group could only appear at regular cardinal stages~$\lambda<\kappa$. Fix such~$\lambda$. To show (inductively) that~$G_\lambda$ is free, we consider the set~$C$ of~$\beta<\kappa$ such that cofinally in~$\beta$ we see stages~$s$ such that $\alpha\in B'_s$. At no stage~$s\in C$ do we twist. The set~$C$ is certainly closed, and $C\cap \lambda$ is cofinal in~$\lambda$ because it is $\kappa$-computable (with parameter smaller than~$\lambda$) and~$C$ is cofinal in~$\kappa$.
\end{remark}


\section{Coding into bases of free groups} 
\label{sec:coding_into_bases}

\Cref{cor:no_upper_bound_for_bases} says that if~$\kappa$ is a successor cardinal then no reasonable oracle suffices to compute a basis for every computable free abelian group. The situation for inaccessible cardinals remains unclear. In this section we tackle the other direction: what can be coded into all bases of some free abelian group? That is, for which sets $D\in 2^\kappa$ can we find a $\kappa$-computable free abelian group, every basis of which $\kappa$-computes~$D$? This is the content of \cref{thm:coding_into_bases}, which we prove in this section. In brief, our results say that:

\begin{itemize}
	\item $\emptyset'$ can always be coded;
	\item an upper bound on the sets that can be coded is the degree of $\Detach{G}$, which is always $\emptyset''$-computable, but sometimes $\emptyset'$-computable;
	\item in many cases, this upper bound can be realised. 
\end{itemize}

\subsection{The limits of coding} 
\label{sub:the_limits_of_coding}

Computing bases of a free group is equivalent to computing clubs through the detachment set. The following is an effective version of \cref{prop:club_and_freeness}. The proof is the same. 

\begin{lemma} \label{lem:computing_bases_and_clubs}
	Let~$\kappa$ be regular and uncountable; let~$G$ be a $\kappa$-computable free abelian group, and let $\bar G = \seq{G_{\alpha}}_{\alpha<\kappa}$ be a $\kappa$-computable filtration of~$G$. 

	The collection of bases of~$G$ and the collection of club subsets of~$\Detach{\bar G}$ are $\kappa$-Medvedev equivalent. That is, there are partial $\kappa$-computable functions $f,g\colon 2^\kappa\to 2^\kappa$ such that for every basis~$B$ of~$G$, $f(B)$ is a club through~$\Detach{\bar G}$; and for every club subset~$C$ of~$\Detach{\bar G}$, $g(C)$ is a basis of~$G$. 
\end{lemma}

The detachment set is the limit of possible coding into bases. 

\begin{theorem} \label{thm:limits_of_detachment}
	Let~$\kappa$ be regular and uncountable, and let~$G$ be a $\kappa$-computable group. For any $X\in 2^\kappa$ which is not $\kappa$-computable from~$\Detach{G}$, there is a basis of~$G$ which does not $\kappa$-compute~$X$. 
\end{theorem}

\begin{proof}
The proof of \cref{thm:limits_of_detachment} uses effective forcing. It is a generalisation of the forcing notion used to shoot a club through a stationary subset of~$\w_1$ \cite{ShootingClub}. Of course working effectively we do not actually extend the universe, so we will use the fact that the detachment set does contain a club (as~$G$ is free). Fix a regular uncountable cardinal~$\kappa$ and a $\kappa$-computable free group~$G$. 

\smallskip

The notion of forcing~$\PP = \PP(G)$ we use is the collection of all closed and bounded subsets of~$\Detach{G}$. The ordering is by end-extension: $D$ extends~$C$ if $D\supseteq C$ and $D\cap {\max C+1} = C$. Note that~$\PP$ is $\kappa$-computable from~$\Detach{G}$. 

While~$\PP$ is not $\kappa$-closed, it satisfies a weaker form of closure which will still allow us to build a sufficiently generic filter in~$\kappa$ many steps. It is \emph{$\kappa$-strategically closed}. This means that playing against an opponent, we have a strategy to stay inside~$\PP$ when alternating extending conditions in plays of length $<\kappa$, as long as we get to play at limit stages. In detail, fix a club~$D\subseteq \Detach{G}$. For $C\in \PP$ let~$g(C) = C\cup \{\min D\setminus (\max C+1)\}$. That is, add to~$C$ the next element of~$D$ beyond $\max C$. If $\gamma<\kappa$ is a limit ordinal and $\seq{C_\alpha}_{\alpha<\gamma}$ is a sequence of extending conditions in~$\PP$ (if $\beta>\alpha$ then $C_\beta$ extends $C_\alpha$ in~$\PP$) such that for any even ordinal $\alpha<\gamma$, $C_{\alpha+1} = g(C_\alpha)$, then letting $C_{<\gamma} = \bigcup_{\alpha<\gamma} C_\alpha$, the condition $C_\gamma= C_{<\gamma} \cup \{\sup C_{<\gamma}\}$ is a condition in~$\PP$ and extends each~$C_\alpha$. The point of course is that $\sup C_{<\gamma}\in \Detach{G}$, as it is in~$D$. In this way we can (within~$L$) build a filter of~$\PP$ meeting any prescribed collection of~$\kappa$ many dense subsets of~$\PP$. 

Fix $X\nle_\kappa \Detach{G}$. Let~$\HH$ be a filter, sufficiently generic over~$X$; let $A = \bigcup \HH$. This is a closed subset of~$\Detach{G}$. One kind of dense set we meet ensures that $A$ is unbounded in~$\kappa$; we can always extend conditions beyond any point below~$\kappa$, as $\Detach{G}$ is unbounded. It remains to show that $X\nle_{\kappa} A$. The argument is similar to the one used for effective Cohen forcing: if $H$ is 1-generic over~$Y$ and~$Y$ is noncomputable then $Y\nle_\Tur H$; here we need~$\Detach{G}$ as a base to compute~$\PP$. Let~$\Phi$ be a $\kappa$-c.e.\ functional, and let $C_0\in \PP$. If there is some~$C\in \PP$ extending~$C_0$ such that $\Phi(C)\perp X$ we take such an extension.
Otherwise, we claim that $C_0$ already forces divergence: there is some $\beta < \kappa$ such that for all $C \in \PP$ extending~$C_0$, $\Phi(C,\beta)\diverge$.
For if not, then using~$\PP$ (and so using $\Detach{G}$) we can $\kappa$-compute~$X$ by ranging over extensions of~$C_0$ and applying~$\Phi$. 
\end{proof}


\subsection{Coding $\cero'$} 
\label{sub:coding_zero_jump}

It is not hard to encode $\emptyset'$. It is possible in all cases, including singular cardinals and~$\w$. 

\begin{theorem} \label{thm:coding_one_jump}
	Let~$\kappa$ be any infinite cardinal. There is a $\kappa$-computable free abelian group, every basis of which computes~$\emptyset'$.
\end{theorem}

\begin{proof}
Begin by constructing a free group on $\kappa$ generators $\{b_\alpha\}$ for $\alpha < \kappa$. If at stage~$s<\kappa$ we see~$\alpha$ entering~$\emptyset'$, at that stage we introduce a new generator equal to~$b_\alpha/2$.

Let~$B$ be a basis of the resulting group~$G$. For each~$\alpha$ there is a finite subset~$B_\alpha$ of~$B$ such that $b_\alpha\in \Span(B_\alpha)$, and such a set can be found computably from~$B$; note that the function $\alpha\mapsto b_\alpha$ is $\kappa$-computable. Because $B_\alpha$ is $P$-independent, $\alpha\in \emptyset'$ if and only if 2 divides $b_\alpha$ in $\Span(B_\alpha)$. Note that even in the case $\kappa=\w$ determining this is computable, looking at the coefficients of~$b_\alpha$ in terms of the generators in~$B_\alpha$. 
\end{proof}

A-priori, for any regular uncountable~$\kappa$, for any~$\kappa$-computable group~$G$, $\Detach{G}$ is $\emptyset''$-computable. \Cref{thm:limits_of_detachment} shows that if~$\kappa$ is a cardinal for which $\Detach{G}$ is $\emptyset'$-computable for every $\kappa$-computable free group~$G$, then \cref{thm:coding_one_jump} is optimal for this~$\kappa$. In this subsection we outline a number of cases in which this holds. 

\begin{proposition} \label{prop:computing_detachment_set_is_sometimes_easy}
	Suppose that~$\kappa$ is a regular uncountable cardinal which is not the successor of a non-weakly-compact regular uncountable cardinal. 
	Then for any $\kappa$-computable free group~$G$, $\Detach{G}$ is $\emptyset'$-computable. 
\end{proposition}

Toward finding the complexity of~$\Detach{G}$, we investigate the complexity of detachment among $\kappa$-finite free groups. Fix a regular uncountable cardinal~$\kappa$. Recall that ``$\kappa$-finite'' just means ``an element of~$L_\kappa$''. We first observe that given a $\kappa$-finite group~$K$ and a subgroup~$H$, we can effectively find a $\kappa$-finite copy of~$K/H$. Using \cref{fact:detachment_and_freeness_of_quotient}, this implies:

\begin{lemma} \label{lem:computing_freeness_and_detachment}
	The collection of $\kappa$-finite free abelian groups is $\kappa$-computably equivalent to the collection of $\kappa$-finite pairs $(K,H)$ such that~$K$ is free and~$H$ is a subgroup of~$K$ which detaches in~$K$. 
\end{lemma}

Note that these sets are $\kappa$-c.e., and so are $\emptyset'$-computable. This implies that for any $\kappa$-computable free group~$G$, $\Detach{G}$ is $\Pi^0_2(L_\kappa)$, and so as promised, $\emptyset''$-computable. 

\smallskip

To prove \cref{prop:computing_detachment_set_is_sometimes_easy} we consider several cases, which together cover all cardinals to which the proposition applies:
\begin{enumerate}
	\item $\kappa$ is the successor of a weakly compact cardinal;
	\item $\kappa = \w_1$;
	\item $\kappa$ is the successor of a singular cardinal;
	\item $\kappa$ is inaccessible.	
\end{enumerate}

For cases~(1)--(3), the proposition follows immediately from the following:

\begin{proposition} \label{prop:computing_freeness_at_successor_of_weakly_compact}
	Suppose that~$\kappa$ falls under cases (1)--(3). 
	Then the collection of $\kappa$-finite free groups is $\kappa$-computable. 
\end{proposition}

\begin{proof}
	First, we consider cases~(1) and~(2). 

	Let~$H$ be a $\kappa$-finite (torsion free, abelian) group. Let~$\kappa^-$ be the cardinal predecessor of~$\kappa$. By adding a copy of $\Z^{\kappa^-}$ we may assume that $|H| = \kappa^{-}$. Effectively we can find a group~$K$ with universe~$\kappa^-$ which is isomorphic to~$H$. Now we use the fact that in both cases, the collection of free groups on~$\kappa^-$ is first-order definable over~$\kappa^-$. If $\kappa = \w_1$, then we know that the collection of free groups on~$\w$ is $\Pi^0_3$; by \cref{thm:identifying_free_groups}, if~$\kappa^-$ is weakly compact, then the collection of free groups on~$\kappa^-$ is $\Pi^0_2(L_\kappa)$. In both cases, whether $(L_{\kappa^-},K)$ satisfies this definition can be effectively computed within~$L_{\kappa}$.

	\medskip

Next we consider case~(3). The proof relies on Shelah's singular compactness theorem \cite{Shelah:SingularCompactness}, see also \cite{Eklof:ShelahSingularCompactness}. Let~$\kappa$ be the successor of a singular cardinal. Shelah's theorem says that (like in the weakly compact case), a group of size~$\kappa^-$ is free if and only if every subgroup of strictly smaller cardinality is free. 

Because the collection of $\kappa^-$-finite free groups is definable over $L_{\kappa^-}$ (it is $\kappa^-$-c.e.), it is $\kappa$-finite. We know that the collection of $\kappa$-finite free groups is $\kappa$-c.e., so it suffices to show it is also $\kappa$-co-c.e. For a $\kappa$-finite group~$K$, the collection of all $\kappa$-finite subgroups of~$K$ cardinality smaller than~$\kappa^-$ is $\kappa$-computable (uniformly in~$K$); for each such group~$H$, we can effectively find a $\kappa^-$-finite group~$\hat H$ isomorphic to~$H$, and then see whether it is free or not.
\end{proof}

We turn to case~(4). The following lemma
will be also useful later, when we discuss singular cardinals.    Recall that if~$H$ is a subgroup of a group~$G$ then we write $[H,G]$ to denote the collection of all subgroups $K\subseteq G$ such that $H\subseteq K$.

\begin{lemma} \label{lem:detachment_in_subgroup_of_same_size}
	Let~$G$ be a free abelian group, and let $H$ be a subgroup of~$G$. Then $H\div G$ if and only if $H\div K$ for all $K\in [H,G]$ such that $|K|=|H|$. 
\end{lemma}

\begin{proof}
	Given \cref{prop:detachment_is_forever}, it suffices to show that if $H\ndiv G$ then there is some $K\in [H,G]$ with $|K|= |H|$ such that $H\ndiv K$. Let~$B$ be a basis of~$G$. Let~$C\subseteq B$ be a subset of size~$|H|$ such that $H\subseteq \Span(C)$; let $K = \Span(C)$. 
\end{proof}

\begin{proof}[Proof of \cref{prop:computing_detachment_set_is_sometimes_easy} in the inaccessible case]
	Suppose that~$\kappa$ is inaccessible. We use \cref{lem:detachment_in_subgroup_of_same_size}. We are given a $\kappa$-finite subgroup~$H$ of~$G$, and want to know whether $H\div G$ or not. First, we use~$\emptyset'$ to find a regular cardinal $\lambda<\kappa$ such that $H\in L_\lambda$. With parameter~$\lambda$ we can computably check, given a $\kappa$-finite subgroup~$K\in [H,G]$ of size smaller than~$\lambda$, whether~$H\div K$ or not: we search for an isomorphism $g$ from~$K$ to a $\lambda$-finite group $g[K]$, and ask whether $g[H]$ detaches in~$g[K]$; since both~$g[K]$ and~$g[H]$ are $\lambda$-finite, the search for this detachment is performed within $L_\lambda$, and so is bounded. 

	Hence, after finding~$\lambda$, we can ask~$\emptyset'$ the following~$\Sigma^0_1$ question, equivalent to~$H\ndiv G$: is there a $\kappa$-finite $K\in[ H,G]$ and an injective function~$g$ from~$K$ into some $\alpha<\lambda$ such that in~$L_\lambda$, $g[H]\ndiv g[K]$?
\end{proof}


\subsection{Coding $\cero''$} 
\label{sub:coding_0_double}

To finish the proof of \cref{thm:coding_into_bases}, we consider the case in which not only $\Detach{G}$ can be made to be equivalent to~$\emptyset''$, but we can code~$\emptyset''$ into all bases of a group; so again in this case our results are tight.

\begin{proposition} \label{prop:detachment_is_complete_at_most_successors}
	Suppose that~$\kappa$ is a successor of a regular uncountable cardinal which is not weakly compact. Then the collection of $\kappa$-finite free abelian groups is $\Sigma^0_1(L_\kappa)$-complete. 
\end{proposition}

\begin{proof}
	Let~$\kappa^-$ be the cardinal predecessor of~$\kappa$. The proposition follows from the fact that the collection of free groups with universe~$\kappa^-$ is $\mathbf{\Sigma}^1_1(L_{\kappa^-})$-complete, and that we can effectively translate $\Sigma^0_1$ questions about $L_{\kappa}$ into $\mathbf{\Sigma^1_1}$ questions about $L_{\kappa^{-}}$. This is not new but we give the details for completeness. The key is the regularity of $\kappa^->\w$, which makes well-foundedness a relatively simple question. 

	As a first step consider first $\Sigma^0_1(L_\kappa)$-questions with no parameters. Let~$\vphi$ be a $\Sigma^0_1$ formula. To find out if $L_\kappa\models \vphi$, we note that this happens if and only if $L_\alpha\models \vphi$ for some $\alpha < \kappa$. (Actually $\alpha<\w_1$, but we are doing this step as a warm-up, and this observation won't help later.) Then $L_\kappa\models \vphi$ if and only if there is some $A\subseteq (\kappa^-)^2$  such that $(\kappa^-,A)$ is a well-founded model of~$\ZF^- + (V=L) + \vphi$. Well foundedness is first-order definable in $(\kappa^-,A)$, as we only quantify over functions from $\w\to \kappa^-$; as $\kappa^-$ is regular, all of these are $\kappa^-$-finite; so this question is $\Sigma^1_1(L_{\kappa^-})$. 

	Now for the general case, we take a $\Sigma^0_1$ formula~$\vphi$ and a parameter $\beta<\kappa$. Effectively, in~$L_\kappa$, we can find a well-ordering~$B$ on~$\kappa^-$ isomorphic to~$\beta$. Our $\Sigma^1_1(B)$ question now asks for some relation~$A$ on~$\kappa^-$ and an embedding of~$(\kappa^-,B)$ into the initial segment of $(\kappa^-,A)$ determined by some $x\in \kappa^-$ such that $(\kappa^-,A)\models \ZF^-+ (V=L) + \vphi(x)$.

	To~$B$ we can add a fixed $C\subseteq \kappa^-$ such that the collection of free abelian groups is $\Sigma^1_1(C)$-complete, and so given~$\vphi(\beta)$ find a group~$G$ on~$\kappa^-$ which is free if and only if $L_\kappa\models \vphi(\beta)$. 
\end{proof}

\begin{corollary} \label{cor:blackbox}
	Suppose that~$\kappa$ is a successor of a regular cardinal which is not weakly compact. There is a partial $\kappa$-computable function which takes as input a $\kappa$-finite free abelian group~$H$ and a $\Sigma^0_1$ formula~$\vphi$ (with parameters in~$L_\kappa$) and outputs a $\kappa$-finite free abelian group~$\BlackBox{H}{\vphi}$ in which~$H$ detaches if and only if~$\vphi$ holds in~$L_\kappa$. 
\end{corollary}

\begin{proof}
	Given~$H$ and~$\vphi$, first use \cref{prop:detachment_is_complete_at_most_successors} to get a $\kappa$-finite group~$K$ (of size~$\kappa^-$) which is free if and only if $\vphi$ holds in~$L_\kappa$. Now the idea is to let~$G = \BlackBox{H}{\vphi}$ be a free extension of~$H$ such that $K\cong G/H$, and then refer to \cref{fact:detachment_and_freeness_of_quotient}. 

	Technically what we do is find a surjection~$f$ from some copy~$G$ of $\Z^{\kappa^-}$ onto~$K$, ensuring that the kernel of~$f$ has size $\kappa^-$; this can be achieved using the freeness of~$\Z^{\kappa^{-}}$. Since a subgroup of a free group is free, the kernel of~$f$ is isomorphic to~$\Z^{\kappa^{-}}$, and so to~$H$. Renaming the elements of~$G$ we can thus assume that $H = \ker f$. 
\end{proof}

\begin{theorem} \label{thm:coding_zero_double}
	Suppose that~$\kappa$ is a successor of a regular cardinal which is not weakly compact. There is a $\kappa$-computable free group, all bases of which compute~$\emptyset''$. 
\end{theorem}

\begin{proof}
	We start with a $\Pi^0_2$-complete set~$P$ such that for any~$X\in 2^\kappa$, if $P$ is $X$-c.e.\ then it is $X$-computable. For example let~$P$ be the join of a $\Pi^0_2$-complete set~$\hat P$ and the collection of all bounded initial segments of~$\emptyset'$: if $X$ enumerates~$P$ then it computes~$\emptyset'$, and then~$\hat P$ is both $X$-c.e.\ and $X$-co-c.e.

	\smallskip

	Our plan is as follows. For each~$\alpha<\kappa$, we will uniformly fix a $\kappa$-finite free group~$H(\alpha)$ and produce a $\kappa$-computable free group~$G(\alpha)\supset H(\alpha)$ such that $H(\alpha)$ detaches in~$G(\alpha)$ if and only if $\alpha\in P$.  Our group will be $G = \bigoplus_{\alpha <\kappa} G(\alpha)$.  We will now argue that the set of $\alpha$ such that $H(\alpha)$ detaches in $G(\alpha)$ is c.e.\ relative to any basis of $G$.

	In the construction of~$G(\alpha)$ we will produce a $\kappa$-computable filtration $\seq{G_s(\alpha)}_{s<\kappa}$, with $G_0(\alpha) = H(\alpha)$.   Having done that, we let $G_s = \bigoplus_{\alpha<s} G_s(\alpha)$ for each $s < \kappa$.  Then $\bar G = \seq{G_s}_{s < \kappa}$ is a filtration of~$G$.  By \cref{lem:computing_bases_and_clubs}, from any basis of~$G$ we effectively obtain a club subset~$C$ of~$\Detach{\bar G}$.
	
	We claim that if $s \in \Detach{\bar G}$ and $\alpha < s$, then $H(\alpha)\div G(\alpha)$ if and only if $H(\alpha)\div G_s(\alpha)$.  For one direction, we have that if $H(\alpha)\div G(\alpha)$, then $H(\alpha)\div G_t(\alpha)$ for every $t < \kappa$.  Conversely, since $\alpha < s$, $G_s(\alpha)\div G_s$ by definition of $G_s$, and since $s \in \Detach{\bar G}$, $G_s\div G$, so by transitivity of detachment, $G_s(\alpha)\div G$.  So if $H(\alpha)\div G_s(\alpha)$, we have that $H(\alpha)\div G$, and so $H(\alpha)$ detaches inside every subgroup of $G$, including $G(\alpha)$.
	
	Since $H(\alpha)\div G_s(\alpha)$ is a $\Sigma^0_1$ relation, we can thus enumerate~$P$ from~$C$ by enumerating all~$\alpha$ such that $H(\alpha)\div G_s(\alpha)$ for some $s \in C$ with $s > \alpha$.
	
	\medskip

	It remains only to uniformly construct the $G(\alpha)$ and their filtrations.  Fix~$\alpha$, and let $H(\alpha)$ be some fixed copy of~$\Int^{\kappa^{-}}$. Fix $\psi$, a bounded-quantifier formula which is the matrix of a definition of~$P$:
	 \[ \beta \in P \iff  L_\kappa \models \forall x \exists y\, \psi(\beta, x,y) .\] 
	  For $s\le\kappa$, let $\ell_s = \ell_s(\alpha)$, the ``length of witnessing'' of the potential membership of $\alpha$ in $P$, to be 
	\[
		\ell_s = \sup \left\{ \gamma<s \,:\,  (\forall x<\gamma) (\exists y<s) \,\,\psi(\alpha,x,y) \right\}.
	\]
	The sequence $\seq{\ell_s}_{s < \kappa}$ is $\kappa$-computable. The sequence $\seq{\ell_s}_{s\le \kappa}$ is non-decreasing and continuous. And $\alpha\in P$ if and only if $\ell_\kappa = \kappa$ if and only if the sequence $\seq{\ell_s}_{s<\kappa}$ is unbounded in~$\kappa$. Further, 
$\alpha\in P$ if and only if for all $s<\kappa$, $\ell_s < \ell_\kappa$; for 
if $\alpha\notin P$, then $\seq{\ell_s}_{s<\kappa}$ is eventually constant; this follows from the fact that $\kappa$ is a regular cardinal: if $\ell_\kappa<\kappa$ then there is some $s<\kappa$ such that for all $\beta<\ell_\kappa$ there is some $y<s$ such that $\psi(\alpha,\beta,y)$ holds. 

	\smallskip
	
	The idea is the following. Given a length~$\ell_s$, we extend $G_s(\alpha)$ to potentially twist~$H(\alpha)$, to be untangled when we discover a greater length $\ell_t>\ell_s$. 

	We give the formal details. For each $\beta<\kappa$, using \cref{cor:blackbox}, let
	\[
		K_\beta = \BlackBox{H(\alpha)}{``\exists s\, (\ell_s > \beta)"}.
	\]
	Each group $K_\beta$ is free, $H(\alpha)\subset K_\beta$, and $H(\alpha)$ detaches in~$K_\beta$ if and only if $\ell_\kappa > \beta$. 
The sequence of groups $\seq{K_\beta}_{\beta<\kappa}$ is $\kappa$-computable. By taking isomorphic copies, we may assume that $K_\gamma \cap K_\beta = H(\alpha)$ if $\beta\ne \gamma$. 

For each $\beta<\ell_\kappa$, since $H(\alpha)$ detaches in~$K_\beta$, we can effectively find a complement $V_\beta$ for $H(\alpha)$ inside~$K_\beta$. 
Note that the function $\beta\mapsto V_\beta$ is only partial $\kappa$-computable (uniformly in~$\alpha$), as the set $\left\{ (\alpha,\beta) \,:\,  \beta<\ell_\kappa(\alpha) \right\}$ is $\kappa$-c.e.\ but not $\kappa$-computable. At each stage $t<\kappa$ we will have found $V_{\beta}$ for all $\beta<\ell_t$.

We now define the sequence of groups $G_t(\alpha)$ for $t\le \kappa$.  
	Let $U = U(\alpha)$ be the set of limit ordinals $t\le \kappa$ such that for all $s<t$, $\ell_s<\ell_t$. 
(Recall that $\alpha\in P$ if and only if $\kappa\in U(\alpha)$.)

	For brevity, for $t\le \kappa$ let
	 $R_t = \Int^t\oplus \bigoplus_{\beta<\ell_t} V_\beta.$ Here $\Int^t$ is some fixed copy of that group such that $\Int^s\subseteq \Int^t$ if $s<t$. 
	 We define:
	\begin{enumerate}
		\item[(i)] If $t\in U$  then $G_t(\alpha) = H(\alpha)\oplus R_t$.
		\item[(ii)] If $t\notin U$ then $G_t(\alpha) = K_{\ell_t} \oplus R_t$.
	\end{enumerate}

	Note that $G_0(\alpha) = H(\alpha)$ as promised. Also note that the function $t\mapsto G_t(\alpha)$  (restricted to $t<\kappa$) is $\kappa$-computable. We need to ensure that this sequence of groups is increasing and continuous. Fix $s< t\le \kappa$;  we show that $G_s(\alpha)\subset G_t(\alpha)$. First note that since $H(\alpha)\subset K_{\ell_s}$, it suffices to show that $K_{\ell_s}\oplus R_s \subset G_t(\alpha)$. There are two cases: 
	\begin{itemize}
		\item If $\ell_t = \ell_s$ then $t\notin U$. In this case the result follows from the fact that $R_s\subset R_t$. 
		\item If $\ell_s < \ell_t$ then $K_{\ell_s}\oplus R_s \subset H(\alpha)\oplus R_t$, and since $H(\alpha)\subset K_{\ell_t}$, we see that $K_{\ell_s}\oplus R_s \subset G_t(\alpha)$ regardless of whether $t\in U$ or not. 
	\end{itemize}
	Finally, suppose that $t\le \kappa$ is a limit ordinal; we need to ensure that $G_t(\alpha) = \bigcup_{s<t} G_s(\alpha)$. But this follows from the fact that $R_t = \bigcup_{s<t} R_s$; we always have $H(\alpha) = G_0(\alpha)\subset \bigcup_{s<t}G_s(\alpha)$, which takes care of the case $t\in U$; if $t\notin U$ then $K_{\ell_t} = K_{\ell_s}$ for some $s<t$ such that $s\notin U$, and so $K_{\ell_t}\subset \bigcup_{s<t}G_s(\alpha)$. 

	\smallskip

	We remark that this static description of the construction, while precise, does mask a little our intentions, which are described dynamically. At a stage $s\notin U$, we have $H(\alpha)$ potentially twisted in $G_s(\alpha)$ (as it is potentially twisted inside $K_{\ell_s}$). It remains this way until we discover some $t>s$ at which we see that $\ell_t > \ell_s$. We then discover that $H(\alpha)$ was not in fact twisted inside $G_s(\alpha)$, and we (potentially) retwist it again inside $G_t(\alpha)$, via $K_{\ell_t}$. 

	\smallskip

	Finally, we need to show that $H(\alpha)$ detaches in~$G(\alpha) = G_\kappa(\alpha)$ if and only if $\alpha\in P$. But we observed that $\alpha\in P$ if and only if $\kappa \in U$. If $\kappa\in U$ then certainly $H(\alpha)\div G_\kappa(\alpha)$. If $\kappa\notin U$ then $\ell_\kappa<\kappa$ and $H(\alpha)$ does not detach in~$K_{\ell_\kappa}$, and as $K_{\ell_\kappa}\div G_\kappa(\alpha)$, we get $H(\alpha) \ndiv G_\kappa(\alpha)$. 
\end{proof}


\subsection{More on~$\kappa$-finite free groups} 
\label{sub:more_on_kappa_finite}

\Cref{prop:computing_detachment_set_is_sometimes_easy,prop:computing_freeness_at_successor_of_weakly_compact} raise a separate question: in general, what is the complexity of the set of $\kappa$-finite free abelian groups? Together with \cref{prop:detachment_is_complete_at_most_successors}, we see that the only case left open is when~$\kappa$ is inaccessible. 

For the following, we generalise the definition of weak truth-table reducibility in terms of bounding the use function. If~$\Phi$ is a $\kappa$-functional and $\Phi(Y)=X$, then for all $\beta<\kappa$ we define the use of this reduction to be the least $\gamma$ such that $(Y\rest{\gamma}, X\rest{\beta})\in \Phi$. We say that $Y$ $\kappa$-wtt computes~$X$ if there is such a functional for which the use function is bounded by a $\kappa$-computable function.


\begin{theorem} \label{prop:kappa_finite_inaccessible_Turing}
	Let~$\kappa$ be inaccessible. Then the collection of $\kappa$-finite free abelian groups $\kappa$-computes $\emptyset'$, but does not $\kappa$-wtt compute~$\emptyset'$ (and so is not 1-complete for the class $\Sigma^0_1(L_\kappa)$). 
\end{theorem}

\begin{proof}
	Let~$\alpha<\kappa$; we want to find out whether $\alpha\in \emptyset'$ or not. We start building an increasing and continuous sequence of groups $\seq{G_\beta}_{\beta\in [\alpha,\alpha^+]}$, always twisting at $\beta\in E$. That is, we start with $G_\alpha$ being trivial. We take unions at limit stages. At successors of successors we add a copy of~$\Int$. Suppose that $\beta>\alpha$ is a limit ordinal and~$G_\beta$ is already defined. We consult our oracle to see if~$G_\beta$ is free. If it is, then~$\beta$ will be singular, and so we can wait for~$\sing{\beta}$ and observe if~$\beta\in E$  or not; if so we twist~$G_\beta$ inside~$G_{\beta+1}$; otherwise we do not. The arguments above show that~$G_\beta$ is free if and only if $\beta<\alpha^+$. So once we see that~$G_\beta$ is not free, we know that~$\beta = \alpha^+$, and we can consult~$L_{\alpha^+}$ to see whether~$\alpha\in \emptyset'$ or not. 

	\smallskip

	Suppose, for a contradiction, that the set of free abelian groups $\kappa$-wtt computes~$\emptyset'$; let $\Psi$ be a reduction. Let~$\lambda<\kappa$ be a successor of a singular cardinal, sufficiently large so that the parameter used to compute~$\Psi$ is in~$L_\lambda$. Then the restriction of~$\Psi$ to~$L_\lambda$ is in fact a $\lambda$-computable reduction of~$\emptyset'(L_\lambda)$ to the set of $\lambda$-finite abelian groups; this contradicts \cref{prop:computing_freeness_at_successor_of_weakly_compact}. 
\end{proof}

\section{Singular cardinals} 
\label{sec:singular_cardinals}

Recall that even when~$\kappa$ is singular, $L_\kappa$ is admissible and~$\kappa$-computability makes sense. When analysing groups with universe~$\kappa$, though, we need to take care, as the notion of filtration is not as robust. In general, if $(L_\kappa,G)$ is not admissible, then it is likely that some bounded subsets of~$G$ generate subgroups which are unbounded. This does not happen when $(L_\kappa,G)$ is admissible (for example, when~$G$ is $\kappa$-computable), as there is a $\kappa$-computable function from $B\times \w$ onto $\Span(B)$. In particular, when~$G$ is $\kappa$-computable, for any cardinal $\lambda < \kappa$, $G\rest{\lambda}$ is a subgroup of~$G$. 

\smallskip

In the absence of well-behaved filtrations we consider the general detachment set, restricted to $\kappa$-finite subgroups. Fix a singular cardinal~$\kappa$ and a $\kappa$-computable group~$G$. First, for a $\kappa$-finite subgroup~$H$ of~$G$, let
\[ [H,G]_\bdd  = [H,G]\cap L_\kappa \]
be the collection of $\kappa$-finite subgroups of~$G$ extending~$H$; and then let~$\FDet(G)$ be the collection of all $\kappa$-finite subgroups~$H$ of~$G$ which detach in every subgroup in $[H,G]_\bdd$.  Let $0$ denote the trivial group.

\begin{lemma} \label{lem:bdd_detachment_set_cofinal}
	Suppose that every $\kappa$-finite subgroup of~$G$ is free. Then
	the detachment set $\FDet(G)$ is cofinal in $[0,G]_\bdd$: every $\kappa$-finite subgroup~$H$ of~$G$ has an extension in $\FDet(G)$. 
\end{lemma}

\begin{proof}
	First, note that non-detachment is witnessed at the same cardinality. That is, if~$H$ is a $\kappa$-finite subgroup of~$G$ which is not in $\FDet(G)$, then there is some $K\in [H,G]_\bdd$ of size~$|H|$ in which~$H$ does not detach. To see this simply apply \cref{lem:detachment_in_subgroup_of_same_size} to~$H$ and the group~$G\rest \lambda$, where $\lambda < \kappa$ is regular and sufficiently large to include~$H$, the parameter used for the computable definition of~$G$, and a~$\kappa$-finite subgroup of~$G$ in which~$H$ does not detach. 

	\smallskip

	Fix some $\kappa$-finite subgroup~$H\notin \FDet(G)$; let $\lambda = |H|^+$. Recall that with parameter~$\lambda$, computing detachment among groups of size~$<\lambda$ is $\kappa$-computable (see the proof of \cref{prop:computing_detachment_set_is_sometimes_easy} in the inaccessible case): to tell whether some $\kappa$-finite group~$K$ detaches in another one~$P$ of size $<\lambda$, find a bijection $g$ from~$P$ to some $\alpha<\lambda$ and then see if in~$L_\lambda$ we can see a complement for $g[H]$ in~$g[P]$. 

	\smallskip

	Suppose, for a contradiction, that~$H$ has no extension in $\FDet(G)$. Now we construct a $\kappa$-computable filtration $\bar H = \seq{H_i}_{i<\lambda}$ of a $\kappa$-finite group~$H_\lambda$ as follows. Starting with $H_0 = H$, given~$H_i$ we find some $H_{i+1}\in [H_i,G]_\bdd$ of size~$|H| = \lambda^-$ in which~$H_i$ does not detach. By the paragraph before, such~$H_{i+1}$ can be found $\kappa$-effectively. This ensures that for all limit $j\le \lambda$, the sequence $\bar{H}\rest{j}$ is $\kappa$-finite and so $H_j = \bigcup_{i<j} H_i$ is~$\kappa$-finite. Here again we crucially used the assumption that~$G$ is $\kappa$-computable. 

	Now we reached our contradiction: by assumption, $H_\lambda$ is free. But $\Detach{\bar H}$ is empty, contradicting the fact that it must contain a club of~$\lambda$ (\cref{prop:club_and_freeness})). 
\end{proof}

\begin{proposition} \label{prop:cofinality_omega_identifying_freeness}
	Suppose that~$\cf(\kappa) = \aleph_0$. Then a $\kappa$-computable group~$G$ is free if and only if every $\kappa$-finite subgroup of~$G$ is free. 
\end{proposition}

Again notice that this is stronger than Shelah's singular compactness theorem, as there are many countable subgroups of~$G$ which are not $\kappa$-finite. 

\begin{proof}
	Suppose that every $\kappa$-finite subgroup of~$G$ is free. Let $\seq{\kappa_n}$ be a cofinal sequence in~$\kappa$. Define a sequence 
	\[
		H_0 \subseteq K_0 \subseteq H_1 \subseteq K_1 \subseteq H_2 \subseteq K_2 \subseteq \dots
	\]
	such that each~$K_i\in \FDet(G)$ and $G\rest {\kappa_n} \subseteq H_n$; for example we can simply let~$H_n$ be the subgroup generated by~$K_{n-1}\cup G\rest{\kappa_n}$. So $G = \bigcup_n K_n$ and each~$K_n$ detaches in~$K_{n+1}$; the familiar process now gives a basis of~$G$. 
\end{proof}

\begin{corollary} \label{cor:simple_singular_bases}
	Suppose that $X\ge_\kappa \emptyset'$ computes a cofinal~$\w$-sequence in~$\kappa$. Then every $\kappa$-computable free group has an $X$-computable basis. 
\end{corollary}

\begin{proof}
	The sequence $\seq{K_n}$ from the proof of \cref{prop:cofinality_omega_identifying_freeness} is computable from~$\emptyset'$ and the sequence $\seq{\kappa_n}$, which is $\emptyset'$-computable; as in \cref{prop:computing_detachment_set_is_sometimes_easy}, $\FDet(G)$ is $\emptyset'$-computable. 
\end{proof}

So for example, if $\kappa < \aleph_\kappa$ (for example $\kappa = \aleph_\w$), then every $\kappa$-computable free group has a $\emptyset'$-computable basis: the set of cardinals is $\emptyset'$-computable, and a cofinal sequence $f\colon \w\to \alpha$ (where $\kappa = \aleph_\alpha$) is $\kappa$-finite. 

\begin{proposition} \label{prop:identifying_free_singular_groups}
	If $\cf(\kappa) = \aleph_0$, then the index-set of the $\kappa$-computable free groups is $\Pi^0_2(L_\kappa)$-complete.
\end{proposition}

\begin{proof}
	Just like the weakly compact case (\cref{prop:weakly_compact:Pi02_complete}); the same construction works.  
\end{proof}


\bibliographystyle{plain}
\bibliography{Free_abelian_groups}

\end{document}